\documentclass[12pt]{amsproc}

\usepackage{amssymb}
\usepackage{amsfonts}
\usepackage[latin1]{inputenc}
\usepackage[mathscr]{euscript}
\usepackage{enumerate}
\usepackage{tkz-graph}
\usetikzlibrary{arrows}
\usetikzlibrary{positioning}

\newtheorem{theorem}{Theorem}[section]
\newtheorem{lemma}[theorem]{Lemma}
\newtheorem{proposition}[theorem]{Proposition}
\newtheorem{corollary}[theorem]{Corollary}

\theoremstyle{definition}
\newtheorem{definition}[theorem]{Definition}
\newtheorem{example}[theorem]{Example}

\theoremstyle{remark}
\newtheorem{remark}[theorem]{Remark}

\numberwithin{equation}{section}

\newcommand{\dmal}[1]{\begin{align*}{#1}\end{align*}}              
\newcommand{\newf}[3]{{#1}:{#2}\longrightarrow {#3}}					 
\newcommand{\newfd}[5]{\begin{array}{cccc}{#1}: &{#2} &\longrightarrow &{#3} \\& {#4} &\longmapsto &{#5}\end{array}}

\newcommand{\la}[1]{\text{$\mathcal{#1}$}}
\newcommand{\lb}[1]{\text{$\mathscr{#1}$}}
\newcommand{\lc}[1]{\text{$\mathbb{#1}$}}

\newcommand{\powerset}[1]{\text{$\lb{P}(#1)$}}								
\newcommand{\eword}{\text{$\omega$}}											

\newcommand{\xia}{\text{$\xi^\alpha$}}											

\newcommand{\uset}[1]{\text{$\uparrow\hspace{-0.1cm}{#1}$}}				
\newcommand{\usetr}[2]{\text{$\uparrow_{\scriptscriptstyle{#2}}\hspace{-0.1cm}{#1}$}} 
\newcommand{\lset}[1]{\text{$\downarrow\hspace{-0.1cm}{#1}$}}				
\newcommand{\lsetr}[2]{\text{$\downarrow_{\scriptscriptstyle{#2}}\hspace{-0.1cm}{#1}$}} 

\newcommand{\dgraphg}[1]{\text{$\lb{#1}$}}									
\newcommand{\dgraphupleg}[3]{\text{$\dgraphg{#1}=(\dgraphg{#1}^0,\dgraphg{#1}^1,#2,#3)$}}	

\newcommand{\dgraph}{\text{$\dgraphg{E}$}}									
\newcommand{\dgraphuple}{\text{$\dgraphupleg{E}{r}{s}$}}					
\newcommand{\alfg}[1]{\text{$\lb{#1}$}}										
\newcommand{\acfg}[1]{\text{$\lb{#1}$}}										
\newcommand{\lbfg}[1]{\text{$\lb{#1}$}}										
\newcommand{\acfrg}[1]{\text{$\acfg{B}_{#1}$}}								

\newcommand{\alf}{\text{$\alfg{A}$}}											
\newcommand{\acf}{\text{$\acfg{B}$}}											
\newcommand{\lbf}{\text{$\lbfg{L}$}}											
\newcommand{\acfra}{\text{$\acfg{B}_{\alpha}$}}								
\newcommand{\lgraphg}[2]{\text{$(\dgraphg{#1},\lbfg{#2})$}}				
\newcommand{\lspaceg}[3]{\text{$(\dgraphg{#1},\lbfg{#2},\acfg{#3})$}}

\newcommand{\lgraph}{\text{$\lgraphg{E}{L}$}}								
\newcommand{\lspace}{\text{$\lspaceg{E}{L}{B}$}}							

\newcommand{\awsetg}[2]{\text{$\lbfg{#1}^{#2}$}}	
\newcommand{\awplus}{\text{$\awsetg{L}{\scriptscriptstyle{\geq 1}}$}}								
\newcommand{\awn}[1]{\text{$\awsetg{L}{#1}$}}								
\newcommand{\awstar}{\text{$\awsetg{L}{\ast}$}}							
\newcommand{\awinf}{\text{$\awsetg{L}{\infty}$}}							
\newcommand{\awleinf}{\text{$\awsetg{L}{\scriptscriptstyle{\leq\infty}}$}}					

\newcommand{\fset}[1]{\text{$\lc{#1}$}}										
\newcommand{\filt}{\text{$\fset{F}$}}											
\newcommand{\filtw}[1]{\text{$\fset{F}_{#1}$}}

\newcommand{\ftight}{\text{$\fset{F}_{tight}$}}										

\newcommand{\ftg}[1]{\text{$\la{#1}$}}											
\newcommand{\ft}{\text{$\ftg{F}$}}												

\begin{document}

\title[Inverse semigroups and labelled spaces]{Inverse semigroups associated with labelled spaces and their tight spectra}

\author[G. Boava \and G. de Castro \and F. Mortari]{Giuliano Boava \and Gilles G. de Castro \and Fernando de L. Mortari}
\address{Departamento de Matemática, Universidade Federal de Santa Catarina, 88040-970 Florianópolis SC, Brazil.}
\email{g.boava@ufsc.br \\ gilles.castro@ufsc.br \\ \newline fernando.mortari@ufsc.br}

\begin{abstract}
The notion of a labelled space was introduced by Bates and Pask in generalizing certain classes of C*-algebras. Motivated by Exel's work on inverse semigroups and combinatorial C*-algebras, we associate each weakly left resolving labelled space with an inverse semigroup, and characterize the tight spectrum of the latter in a way that is reminiscent of the description of the boundary path space of a directed graph.
\end{abstract}

\maketitle

\section{Introduction}
When studying C*-algebras, it is often useful do describe them as C*-algebras associated with groupoids \cite{MR584266}. For certain classes of C*-algebras, such descriptions have been achieved with the assistance of inverse semigroups, most notably with the work of Paterson on \emph{graph C*-algebras} \cite{MR1962477} and of others inspired by his approach \cite{MR2419901,MR2184052,MR2457327}. The desired groupoid appears there as the groupoid of germs for an action of an inverse semigroup (defined from the original object considered, e.g. a directed graph in \cite{MR1962477}) on a certain topological space. A natural candidate for a topological space for the inverse semigroup to act on is the set of characters defined on the idempotents of the semigroup;  however, the C*-algebra constructed from the groupoid of germs that comes from this action is not usually isomorphic to the original C*-algebra.

What was done in the examples studied to somehow fix this is to restrict the space of characters using information from the original object; one of the main goals of Exel in \cite{MR2419901} is to investigate whether this restriction can be done using only information contained in the inverse semigroup, without returning to the original object. He showed that it is possible to do so, by considering only those characters that are \emph{tight}, that is, using the \emph{tight spectrum} of the idempotent semilattice.

Many of the C*-algebras studied in the above-mentioned work can be realized as C*-algebras associated with \emph{labelled graphs} (or more precisely, \emph{labelled spaces}), as defined by Bates and Pask \cite{MR2304922} and further studied with collaborators \cite{2012arXiv1203.3072B,MR2542653,2011arXiv1106.1484B}. Examples include C*-algebras associated with ultragraphs and the Matsumoto algebras of (two-sided) subshifts \cite{MR2304922}, as well as the Carlsen-Matsumoto algebra of a one-sided shift \cite{2012arXiv1203.3072B}.

The main goal of this paper is taking the first step in bringing inverse semigroup theory to the study of labelled spaces in general. Motivated by the inverse semigroup treatment given to graph, ultragraph and higher rank graph C*-algebras in \cite{MR2419901,MR2184052,MR2457327,MR1962477}, we construct an inverse semigroup from a labelled space that is \emph{weakly left resolving}, and proceed to give a description of the tight spectrum for this inverse semigroup.  This is achieved through a study and complete characterization of filters in the idempotent semilattice of this inverse semigroup. The resulting description resembles the characterization of the \emph{boundary path space} of a directed graph as the set consisting of all infinite paths on the graph, and of all finite paths that end in a singular vertex \cite{MR3119197}; indeed we show that, for certain classes of labelled spaces, we arrive exactly at the boundary path space of the graph, but that this may not always be the case. Although motivated by the theory of C*-algebras, we work strictly in the realm of inverse semigroup theory.

This paper is organized as follows: Section 2 is a review of some of the necessary terminology and results on labelled spaces and filters and characters in semilattices. Section 3 deals with the construction of an inverse semigroup associated with a labelled space. Sections 4, 5 and 6 give  a description of the filters, ultrafilters and tight filters, respectively, on the semilattice of idempotents of the inverse semigroup defined in Section 3, and examples are studied.

\section{Preliminaries}

\subsection{Labelled graphs}
\label{subsection:labelled.graphs}

\begin{definition}
	A \emph{(directed) graph} $\dgraphuple$ consists of nonempty countable sets $\dgraph^0$, $\dgraph^1$ and functions $r,s:\dgraph^1\to \dgraph^0$; an element of $\dgraph^0$ is called a \emph{vertex} of the graph, and an element of $\dgraph^1$ is called an \emph{edge}. For an edge $e$, we say that $r(e)$ is the \emph{range} of $e$ and $s(e)$ is the \emph{source} of $e$.
\end{definition}

\begin{definition}A vertex $v$ in a graph $\dgraph$ is called a \emph{source} if $r^{-1}(v)=\emptyset$, a \emph{sink} if $s^{-1}(v)=\emptyset$, and an \emph{infinite emitter} if $s^{-1}(v)$ is an infinite set. The vertex $v$ is \emph{singular} if it is either a sink or an infinite emitter, and \emph{regular} otherwise.
\end{definition}

\begin{definition}
	A \emph{path of length $n$} in a graph $\dgraph$ is a sequence $\lambda=\lambda_1\lambda_2\ldots\lambda_n$ of edges such that $r(\lambda_i)=s(\lambda_{i+1})$ for all $i=1,\ldots,n-1$. We write $|\lambda|=n$ for the length of $\lambda$ and regard vertices as paths of length $0$. We denote by $\dgraph^n$ the set of all paths of length $n$ and $\dgraph^{\ast}=\cup_{n\geq 0}\dgraph^n$. We extend the range and source maps to $\dgraph^{\ast}$ by defining $s(\lambda)=s(\lambda_1)$ and $r(\lambda)=r(\lambda_n)$ if $n\geq 2$ and $s(v)=v=r(v)$ for $n=0$. Similarly, we define a \emph{path of infinite length} (or an \emph{infinite path}) as an infinite sequence $\lambda=\lambda_1\lambda_2\ldots$ of edges such that $r(\lambda_i)=s(\lambda_{i+1})$ for all $i\geq 1$. The set of all infinite paths will be denoted by $\dgraph^{\infty}$.
\end{definition}

\begin{definition}
	Given a nonempty set $\alf$, called an \emph{alphabet}, and whose elements are referred to as \emph{letters}, a \emph{labelled graph} $\lgraph$ over the alphabet $\alf$ consists of a graph $\dgraphuple$ and a map $\lbf:\dgraph^1\to\alf$, called the \emph{labelling map}. 
\end{definition}

Without loss of generality, we may assume that $\lbf$ is onto.

Let $\alf^{\ast}$ be the set of all finite words over $\alf$, including the empty word $\eword$, and extend the map $\lbf$ to a map $\lbf:\dgraph^n\to\alf^{\ast}$ defined by $\lbf(\lambda)=\lbf(\lambda_1)\ldots\lbf(\lambda_n)$. Similarly, one can define the map $\lbf:\dgraph^{\infty}\to\alf^{\infty}$, where $\alf^{\infty}$ is the set of infinite countable words over \alf.

Given $A\subseteq\powerset{\dgraph^0}$ (where $\powerset{\dgraph^0}$ is the power set of $\dgraph^0$) define
\begin{equation}\label{eqn:lae1}
	\lbf(A\dgraph^1)=\{\lbf(e)\ |\ e\in\dgraph^1\ \mbox{and}\ s(e)\in A\}.
\end{equation}

\begin{definition}
	The elements of $\awn{n}:=\lbf(\dgraph^n)$ are called \emph{labelled paths of length $n$} and the elements of $\awinf:=\lbf(\dgraph^{\infty})$ are called \emph{labelled paths of infinite length}. If $\alpha$ is a labelled path, a path on the graph $\lambda$ such that $\lbf(\lambda)=\alpha$ is called a representative of $\alpha$. Since any two representatives of $\alpha$ have the same length, one can define the \emph{length of} $\alpha$, denoted by $|\alpha|$, as the length of any one of its representatives. We also consider $\eword$ as a labelled path, with $|\eword|=0$.  If $1\leq i\leq j\leq |\alpha|$, let $\alpha_{i,j}=\alpha_i\alpha_{i+1}\ldots\alpha_{j}$ if $j<\infty$ and $\alpha_{i,j}=\alpha_i\alpha_{i+1}\ldots$ if $j=\infty$. If $j<i$ set $\alpha_{i,j}=\eword$. The set $\awplus=\cup_{n\geq 1}\awn{n}$ is the set of all labelled paths of positive finite length. We also define $\awstar=\{\eword\}\cup\awplus$, and $\awleinf=\awstar\cup\awinf$.
\end{definition}

Since labelled paths are words over an alphabet, we can concatenate two labelled paths $\alpha$ and $\beta$ to obtain a new word $\alpha\beta$. However, it is not always true that $\alpha\beta$ is also a labelled path, for there may not be a path on the graph whose label is $\alpha\beta$.

\begin{definition}
	Given two labelled paths $\alpha,\beta$, we say that $\alpha$ is a \emph{beginning of} $\beta$ if $\beta=\alpha\beta'$ for some labelled path $\beta'$. We say that $\alpha$ and $\beta$ are \emph{comparable} if $\alpha$ is a a beginning of $\beta$ or $\beta$ is a beginning of $\alpha$.
\end{definition}

\begin{definition}
	For $\alpha\in\awstar$ and $A\in\powerset{\dgraph^0}$, the \emph{relative range of $\alpha$ with respect to} $A$, denoted by $r(A,\alpha)$, is the set
	\[r(A,\alpha)=\{r(\lambda)\ |\ \lambda\in\dgraph^{\ast},\ \lbf(\lambda)=\alpha,\ s(\lambda)\in A\}\]
	if $\alpha\in\awplus$ and $r(A,\eword)=A$ if $\alpha=\eword$. The \emph{range of $\alpha$}, denoted by $r(\alpha)$, is the set \[r(\alpha)=r(\dgraph^0,\alpha).\]
	For $\alpha\in\awplus$ we also define the \emph{source of $\alpha$} as the set \[s(\alpha)=\{s(\lambda)\in\dgraph^0\ |\ \lbf(\lambda)=\alpha\}.\]
\end{definition}

Observe that, in particular, $r(\eword)=\dgraph^0$, and if $\alpha\in\awplus$ then  $r(\alpha)=\{r(\lambda)\in\dgraph^0\ |\ \lbf(\lambda)=\alpha\}$. The definitions above give maps $s,r:\awplus\to\powerset{\dgraph^0}$. Also, if $\alpha,\beta\in\awstar$ are such that $\alpha\beta\in\awstar$ then $r(r(A,\alpha),\beta)=r(A,\alpha\beta)$. Finally, for $A,B\in\powerset{\dgraph^0}$ and $\alpha\in\awstar$, it holds that $r(A\cup B,\alpha)=r(A,\alpha)\cup r(B,\alpha)$.

The following notions were introduced by Bates and Pask in \cite{MR2304922}.

\begin{definition}\label{def:labelled.space}
	Let $\lgraph$ be a labelled graph and $\acf\subseteq\powerset{\dgraph^0}$. We say that $\acf$ is \emph{closed under relative ranges} if $r(A,\alpha)\in\acf$ for all $A\in\acf$ and all $\alpha\in\awstar$. If additionally $\acf$ is closed under finite intersections and unions and contains all $r(\alpha)$ for $\alpha\in\awplus$, we say that $\acf$ is \emph{accommodating} for $\lgraph$ and in this case we say that $\lspace$ is a \emph{labelled space}.
\end{definition}

\begin{definition}
	We say that a labelled space $\lspace$ is \emph{weakly left resolving} if for all $A,B\in\acf$ and all $\alpha\in\awplus$ we have $r(A\cap B,\alpha)=r(A,\alpha)\cap r(B,\alpha)$.
\end{definition}

Observe that in a weakly left resolving labelled space, if $A,B\in\acf$ are disjoint then for all $\alpha\in\awstar$, $r(A,\alpha)$ and $r(B,\alpha)$ are disjoint.

\begin{definition}
	For a given $\alpha\in\awstar$, let $\acfra:=\acf\cap\powerset{r(\alpha)}$.
\end{definition}

Note that $\acfrg{\eword}=\acf$. These sets are fundamental to the analysis that follows.

\begin{remark}
	\label{rmk:b.alpha.boolean.algebra}
	Throughout most of the paper, we assume that the accommodating family for the labelled space $\lspace$ is also closed under relative complements; in this case the set $\acfra$ is a Boolean algebra for each $\alpha\in\awplus$, and  $\acfrg{\eword}=\acf$ is a generalized Boolean algebra as in \cite{MR1507106}. By the duality given by Stone's representation theorem  there is a topological space associated with each $\acfra$ with $\alpha\in\awstar$, which we denote by $X_{\alpha}$, consisting of the set of ultrafilters in $\acfra$ (for details, see \cite{MR0161813} for instance).

\end{remark}

\begin{lemma}\label{lemma:wlr.closedcomp.relrangeisnice}
	Let $\lspace$ be a weakly left resolving labelled space such that $\acf$ is closed under relative complements. If $A,B\in\acf$ and $\alpha\in\awstar$ then $r(A\setminus B,\alpha)=r(A,\alpha)\setminus r(B,\alpha)$.
\end{lemma}

\begin{proof} Let $A,B\in\acf$. Clearly,
	\[r(A\setminus B,\eword)=A\setminus B=r(A,\eword)\setminus r(B,\eword).\]
	
	For $\alpha\in\awplus$, since $A=(A\setminus B)\sqcup(A\cap B)$ (where $\sqcup$ stands for the union of disjoint sets) and the labelled space is weakly left resolving,
	\[r(A,\alpha)=r((A\setminus B)\sqcup(A\cap B),\alpha)=r(A\setminus B,\alpha)\sqcup r(A\cap B,\alpha)\]
	so that
	\dmal{r(A\setminus B,\alpha) & =r(A,\alpha)\setminus (r(A\cap B,\alpha)) \\ & =r(A,\alpha)\setminus(r(A,\alpha)\cap r(B,\alpha)) \\ & =r(A,\alpha)\setminus r(B,\alpha).}
\end{proof}

\subsection{Filters and characters}
\label{subsection:filters.and.characters}

Following \cite{MR2419901}, we recall some basic definitions and properties about filters and characters.

\begin{definition}
	\label{def:filter}
	Let $P$ be a partially ordered set with least element 0. A \emph{filter} in $P$ is a nonempty subset $\xi\subseteq P$ such that
	\begin{enumerate}[(i)]
		\item $0\notin\xi$;
		\item if $x\in\xi$ and $x\leq y$, then $y\in\xi$;
		\item if $x,y\in\xi$, there exists $z\in\xi$ such that $z\leq x$ and $z\leq y$.
	\end{enumerate}
	An \emph{ultrafilter} is a filter which is not properly contained in any filter.
\end{definition}

\begin{remark}
	\label{remark:filter.in.semilattice}
	When $P$ is a (meet) semilattice, condition (iii) may be replaced by $x\wedge y\in\xi$ if $x,y\in\xi$.
\end{remark}

Let $P$ be a partially ordered set with least element $0$ and $x\in P$. We denote the principal filter generated by $x$ by $\uset{x}$, that is,
\[\uset{x}=\{y\in P \ | \ x\leq y\}.\]
Similarly, the principal ideal generated by $x$ is
\[\lset{x}=\{y\in P \ | \ y\leq x\}.\]
When $X$ and $Y$ are subsets of $P$, we use
\[\uset{X} = \bigcup_{x\in X}\uset{x} = \{y\in P \ | \ x\leq y \ \mbox{for some} \ x\in X\},\]
and $\usetr{X}{Y} = Y\cap\uset{X}$; the sets $\usetr{x}{Y}$, $\lsetr{x}{Y}$, $\lset{X}$ and $\lsetr{X}{Y}$ are defined analogously. We remark that $\uset{X}$ and $\usetr{X}{Y}$ are not filters, in general. The next proposition characterizes the subsets $X$ of $P$ for which $\uset{X}$ is a filter.

\begin{proposition}
	\label{prop:subsets.that.generate.filters}
	Let $P$ be a partially ordered set with least element 0 and $X$ be a nonempty subset of $P$. Then $\uset{X}$ is a filter if and only if $0\notin X$ and for every $x,y\in X$, there exists $z\in X$ such that $z\leq x$ and $z\leq y$. In particular, if $P$ is a semilattice then $\uset{X}$ is a filter if and only if $0\notin X$ and for every $x,y\in X$ one has $x\wedge y\in X$.
\end{proposition}
\begin{proof}
	Trivial.
\end{proof}

\begin{proposition}[\cite{MR2419901}, Lemma 12.3]
	\label{prop:ultrafilter.intersection}
	Let $E$ be a semilattice with 0. A filter $\xi$ in $E$ is an ultrafilter if and only if
	\[\{y\in E \ | \ y\wedge x\neq0 \ \forall\,x\in \xi\}\subseteq \xi.\]
\end{proposition}

When $E$ is a Boolean algebra, we have another characterization of ultrafilters.

\begin{proposition}[\cite{MR648287}, Theorem IV.3.12]
	\label{prop:ultrafilter.in.boolean.algebra}
	Let $E$ be a Boolean algebra. A filter $\xi$ in $E$ is an ultrafilter if and only if for every $x\in E$, exactly one of $x$ and $\neg x$ is in $\xi$.
\end{proposition}

\begin{definition}
	\label{def:character}
	Let $E$ be a semilattice with $0$. A \emph{character} of $E$ is a nonzero function $\phi$ from $E$ to the Boolean algebra $\{0,1\}$ such that
	\begin{enumerate}[(a)]
		\item $\phi(0)=0$;
		\item $\phi(x\wedge y)=\phi(x)\wedge\phi(y)$, for all $x,y\in E$.
	\end{enumerate}
	The set of all characters of $E$ is denoted by $\hat{E}_0$ and we endow $\hat{E}_0$ with the topology of pointwise convergence.
\end{definition}

The next definitions are adaptations of \cite{MR2419901}.

\begin{definition}
	\label{def:cover}
	Let $E$ be a semilattice with $0$ and $x\in E$. A set $Z\subseteq \lset{x}$ is said to be a \emph{cover} for $x$ if for all nonzero $y\in\lset{x}$, there exists $z\in Z$ such that $z\wedge x\neq 0$.
\end{definition}

\begin{definition}
	\label{def:character.tight}
	Let $E$ be a semilattice with $0$. A character $\phi$ of $E$ is \emph{tight} if for every $x\in E$ and every finite cover $Z$ for $x$, we have
	\[\bigvee_{z\in Z}\phi(z)=\phi(x).\]
	The set of all tight characters of $E$ is denoted by $\hat{E}_{tight}$, and called the \emph{tight spectrum} of $E$.
\end{definition}

We can associate each filter $\xi$ in a semilattice $E$ with a character $\phi_{\xi}$ of $E$ given by
\[\phi_{\xi}(x)=\left\{\begin{array}{lll} 1, & & \mbox{if} \ x\in E, \\ 0, & & \mbox{otherwise.} \end{array} \right.\]
Conversely, when $\phi$ is a character, $\xi_{\phi} = \{x\in E \ | \ \phi(x)=1\}$ is a filter in $E$. Clearly, these maps give a bijection between $\hat{E}_0$ and the set of filters in $E$. We denote by $\hat{E}_{\infty}$ the set of all characters $\phi$ of $E$ such that $\xi_{\phi}$ is an ultrafilter, and a filter $\xi$ in $E$ is said to be \emph{tight} if $\phi_{\xi}$ is a tight character.

\begin{proposition}[\cite{MR2419901}, Proposition 12.7]
	\label{prop:ultrafilter.is.tight}
	Every ultrafilter is tight, that is, $\hat{E}_{\infty}\subseteq\hat{E}_{tight}$.
\end{proposition}

Definition \ref{def:character.tight} can be rephrased to characterize tight filters, as in the following result.

\begin{proposition}
	\label{prop:tight.filter.characterization}
	A filter $\xi$ in $E$ is tight if and only if for every $x\in\xi$ and every finite cover $Z$ for $x$, one has $Z\cap\xi\neq\emptyset$.
\end{proposition}

When $E$ is a Boolean algebra, tight filters are easily described:

\begin{proposition}[\cite{MR2419901}, Proposition 11.9]
	\label{prop:tight.filters.in.boolean.algebras}
	Let $E$ be a Boolean algebra and $\xi$ a filter in $E$. The following are equivalent:
	\begin{enumerate}[(i)]
		\item $\xi$ is tight;
		\item $\xi$ is an ultrafilter;
		\item $\phi_{\xi}$ is a morphism of Boolean algebras.
	\end{enumerate}
\end{proposition}

The next theorem is the main result relating ultrafilters and tight characters.

\begin{theorem}[\cite{MR2419901}, Theorem 12.9]
	\label{thm:closure.of.ultrafilters}
	Let $E$ be a semilattice with $0$. Then $\hat{E}_{tight}$ is the closure of $\hat{E}_{\infty}$ in $\hat{E}_0$.
\end{theorem}

The convergence of characters in $\hat{E}_0$ can be translated to the context of filters; thus, a net $\{\xi_{\lambda}\}_{\lambda\in\Lambda}$ of filters in $E$ converges to a filter $\xi$ if for each $x\in E$, there exists $\lambda_0\in\Lambda$ such that for all $\lambda\geq\lambda_0$, $x\in\xi_{\lambda}$ if and only if $x\in\xi$.

\section{Definition of the inverse semigroup}\label{section:inv.semigroup.def}
For this section let us fix a labelled space $\lspace$, and consider the set $T$ of all triples $(\alpha,A,\beta)\in \awstar\times\acf\times\awstar$ for which $A\in\acfra\cap\acfrg{\beta}$, together with an extra element, say $z$. A binary operation on $T$ is defined as follows: for all $t\in T$ define $zt=tz=z$, and for elements $(\alpha,A,\beta),(\gamma,B,\delta)$ in $\awstar\times\acf\times\awstar$ put
\[(\alpha,A,\beta)(\gamma,B,\delta)=\left\{\begin{array}{ll}
(\alpha\gamma ',r(A,\gamma ')\cap B,\delta), & \text{if}\ \  \gamma=\beta\gamma ',\\
(\alpha,A\cap r(B,\beta '),\delta\beta '), & \text{if}\ \  \beta=\gamma\beta ',\\
z, & \text{otherwise}.
\end{array}\right. \]
Thus, nontrivial products occur only if $\beta$ and $\gamma$ are comparable. This product is indeed well-defined, for one can see for instance in the first case that $A\in\acfrg{\alpha}$ implies $r(A,\gamma ')\in\acfrg{\alpha\gamma '}$, which gives $r(A,\gamma ')\cap B\in\acfrg{\alpha\gamma '}$. Similarly, $B\in\acfrg{\delta}$ implies $r(A,\gamma ')\cap B\in\acfrg{\delta}$, establishing the condition $r(A,\gamma ')\cap B\in\acfrg{\alpha\gamma '}\cap\acfrg{\delta}.$

\begin{remark}\label{remark:product.same.middle}
	Note that if $\beta=\gamma$ then $\gamma '=\beta '=\eword$ and, therefore, $r(A,\gamma ')=A, r(B,\beta ')=B$, whence \[(\alpha,A,\beta)(\beta,B,\delta)=(\alpha,A\cap B,\delta).\]
\end{remark}

\begin{proposition}\label{prop:semigroup.T}
	Suppose that the labelled space $\lspace$ is weakly left resolving; then, the set $T$ together with the operation above is a semigroup with zero element given by $z$.
\end{proposition}
\begin{proof}
	It must be shown that the operation is associative. Let $s,t,u\in T$ be given. Clearly associativity holds if any of these equals $z$, so suppose they are all in $\awstar\times\acf\times\awstar$, say \[s=(\alpha,A,\beta),t=(\gamma,B,\delta),u=(\mu,C,\nu).\]
	Begin by considering the case $st=z$ and $tu\neq z$, observing that $st=z$ implies $\beta$ and $\gamma$ are not comparable, while $tu\neq z$ says the first entry in $tu$ is either $\gamma$ (if $\delta=\mu\delta '$) or $\gamma\mu '$ (if $\mu=\delta\mu '$); but $\beta$ is not comparable with any of these, so $s(tu)=z=zu=(st)u$.
	
	Evidently the case $st\neq z$ and $tu=z$ follows similarly, and the case $st=tu=z$ is immediate, so it remains to consider $st\neq z$ and $tu\neq z$. There are four cases to investigate, according to how the labels $\beta,\gamma,\delta,\mu$ compare.
	
	If $\gamma=\beta\gamma '$ and $\mu=\delta\mu '$, note that $\gamma\mu '=\beta\gamma '\mu '$ and
	\dmal{(st)u&=(\alpha\gamma ',r(A,\gamma ')\cap B,\delta)(\mu,C,\nu)\\&=(\alpha\gamma '\mu ',r(r(A,\gamma ')\cap B,\mu')\cap C,\nu)\\&=(\alpha\gamma '\mu ',[r(r(A,\gamma '),\mu ')\cap r(B,\mu')]\cap C,\nu)\\&=(\alpha\gamma '\mu ',r(A,\gamma '\mu ')\cap [r(B,\mu')\cap C],\nu)\\&=(\alpha,A,\beta)(\gamma\mu ',r(B,\mu ')\cap C,\nu)\\&=s(tu),} where in the third equality the weakly left resolving hypothesis was used.
	
	The case $\beta=\gamma\beta '$ and $\delta=\mu\delta '$ is analogous, and the case $\gamma=\beta\gamma '$ and $\delta=\mu\delta '$ follows directly from the definition, by comparing the two sides.
	
	Finally, if $\beta=\gamma\beta '$ and $\mu=\delta\mu '$, one has $st=(\alpha,A\cap r(B,\beta '),\delta\beta ')$ and $tu=(\gamma\mu ',r(B,\mu ')\cap C,\nu)$. To evaluate $(st)u$, compare $\delta\beta '$ with $\mu$.
	
	Suppose $\mu$ begins with $\delta\beta '$, say $\mu=\delta\beta '\mu ''$; since $\mu=\delta\mu '$, this is the same as $\mu '=\beta '\mu ''$ and, in particular, $\gamma\mu '=\gamma(\beta '\mu '')=(\gamma\beta ')\mu ''=\beta\mu ''$, so that
	\dmal{(st)u&=(\alpha,A\cap r(B,\beta '),\delta\beta ')(\mu,C,\nu)\\&=(\alpha\mu '',r(A\cap r(B,\beta '),\mu '')\cap C,\nu)\\&=(\alpha\mu '',[r(A,\mu '')\cap r(r(B,\beta '),\mu '')]\cap C,\nu)\\&=(\alpha\mu '',[r(A,\mu '')\cap r(B,\beta '\mu '')]\cap C,\nu)\\&=(\alpha\mu '',r(A,\mu '')\cap [r(B,\mu ')\cap C],\nu)\\&=(\alpha,A,\beta)(\gamma\mu ',r(B,\mu ')\cap C,\nu)\\&=s(tu).}
	
	The case where $\delta\beta '$ begins with $\mu=\delta\mu '$ is analogous, and the case in which they are not comparable is trivial, since then both $(st)u$ and $s(tu)$ equal $z$.
\end{proof}

The next step is to obtain a set $S$ by identifying with $z$ all elements in $T$ of the form $(\alpha,\emptyset,\beta)$, collapsing them all to a single element, represented by $0$, and leaving the other elements in $T$ as they are. The operation on the semigroup $T$ passes to this quotient $S$, turning it into a semigroup with zero element $0$.

In fact, to show the operation is well defined it is enough to see it is well behaved when operating with the zero class, which is easily done: if $s=(\alpha,\emptyset,\beta)$ and $t=(\gamma,B,\delta)$, note that

\begin{itemize}
	\item if $\beta$ and $\gamma$ are not comparable, then $st=z$;
	\item if $\gamma=\beta\gamma '$, then \dmal{st&=(\alpha\gamma ',r(\emptyset,\gamma ')\cap B,\delta)=(\alpha\gamma ',\emptyset\cap B,\delta)=(\alpha\gamma ',\emptyset,\delta);}
	\item if $\beta=\gamma\beta '$ then $st=(\alpha,\emptyset,\delta\beta ')$.
\end{itemize}
Thus, in all three cases the class of $st$ in the quotient equals $0$, and the same can be said for the class of $ts$.

\begin{remark}\label{remark:comparable.product.zero}
	Given $s=(\alpha,A,\beta)$ and $t=(\gamma,B,\delta)$ in $S$, their product $st$ may be equal to $0$ even when $\beta$ and $\gamma$ are comparable, for it may be the case that $r(A,\gamma ')\cap B=\emptyset$ (when $\gamma=\beta\gamma '$) or $A\cap r(B,\beta ')=\emptyset$ (when $\beta=\gamma\beta '$).
\end{remark}

\begin{proposition}\label{prop:inv.semigroup.S}
	Suppose the labelled space $\lspace$ is weakly left resolving. Then $S$ is an inverse semigroup with zero element $0$.
\end{proposition}
\begin{proof}
	Let a nonzero $s=(\alpha,A,\beta)\in S$ be given. Then the element $s^\ast$ defined by $(\beta,A,\alpha)$ is such that $s^\ast s=(\beta,A,\beta)$ (see Remark \ref{remark:product.same.middle}), giving \[ss^\ast s=s(s^\ast s)=(\alpha,A,\beta)(\beta,A,\beta)=(\alpha,A,\alpha)=s,\] and also $s^\ast ss^\ast=s^\ast$.
	
	To prove uniqueness of the inverse, suppose $t\in S$ is such that $sts=s$ and $tst=t$, noting that the first of these identities implies $t\neq 0$. Writing $t=(\gamma,B,\delta)$, from $st\neq 0$ it follows that $\beta$ and $\gamma$ are comparable, and from $ts\neq 0$ one sees $\delta$ and $\alpha$ are comparable.
	
	Let us show that $\beta=\gamma$. Indeed, if $\gamma=\beta\gamma '$ with $\gamma '\neq\eword$, by definition one has
	\[s=(st)s=\left\{\begin{array}{ll}
	(\alpha\gamma '\alpha ',r(r(A,\gamma ')\cap B,\alpha ')\cap A,\beta), & \text{if}\ \  \alpha=\delta\alpha ',\\
	(\alpha\gamma ',[r(A,\gamma ')\cap B]\cap r(A,\delta '),\beta\delta '), & \text{if}\ \  \delta=\alpha\delta ',\end{array}\right.\]
	a contradiction since both first entries, $\alpha\gamma '\alpha '$ and $\alpha\gamma '$, are not equal to $\alpha$. The case $\beta=\gamma\beta '$ with $\beta '\neq\eword$ reaches a similar contradiction, using $t=t(st)$ and looking at the resulting third entries instead. An analogous argument establishes that $\alpha=\delta$, and so it can be concluded that $t=(\beta,B,\alpha)$.
	
	Finally, $s=(st)s$ and $t=(ts)t$ now give $A\cap B=A$ and $B\cap A=B$, whence $A=B$, that is, $t=s^\ast$.
\end{proof}

In what follows we describe the filters, ultrafilters and tight filters in $E(S)$.

\section{Filters in $E(S)$}
\label{section:filters.and.characters.in.es}

Recall that if $T$ is an inverse semigroup then its set of idempotents $E(T)$ is a semilattice; the meet between two idempotents $p$ and $q$ is $p\wedge q=pq=qp$ (idempotents in an inverse semigroup commute), and the induced order is given by $p\leq q$ when $pq=p$.

Let $\lspace$ be a labelled space which is weakly left resolving and consider the inverse semigroup $S$ as in the previous section. In this case, the set $E(S)$ of idempotents is a semilattice with $0$. In this section, we describe the elements, order and filters in $E(S)$.

Let $p=(\alpha, A, \beta)$ be a nonzero element in $S$. If $p$ is an idempotent, then either $\beta$ is a beginning of $\alpha$ or $\alpha$ is a beginning of $\beta$. In first case, writing $\alpha=\beta\alpha'$,
\[p^2=(\alpha, A, \beta)(\alpha, A, \beta) = (\alpha\alpha', r(A,\alpha')\cap A, \beta).\]
Since $p^2=p$, we must have $\alpha'=\eword$, hence $\alpha=\beta$. Analogously, when $\alpha$ is a beginning of $\beta$, we also conclude that $\alpha=\beta$. On the other hand, every element in $S$ of the form $(\alpha, A, \alpha)$ is an idempotent. Thus,
\[E(S)=\{(\alpha, A, \alpha) \ | \ \alpha\in\awstar \ \mbox{and} \ A\in\acfra\}\cup\{0\}.\]

\begin{proposition}
	\label{prop:order.in.es}
	Let $p=(\alpha, A, \alpha)$ and $q=(\beta, B, \beta)$ be nonzero elements in $E(S)$. Then $p\leq q$ if and only if $\alpha=\beta\alpha'$ and $A\subseteq r(B,\alpha')$.
\end{proposition}
\begin{proof}
	Suppose $p\leq q$; by definition this occurs if and only if $(\alpha, A, \alpha)(\beta, B, \beta)=(\alpha, A, \alpha)$, which gives $\alpha=\beta\alpha'$. The multiplication results in $(\alpha, A\cap r(B,\alpha'), \beta\alpha')$ from where it can be deduced that $A\subseteq r(B,\alpha')$. 
	
	For the converse, it is clear that if $\alpha=\beta\alpha'$ and $A\subseteq r(B,\alpha')$, then $pq=p$.
\end{proof} 

\begin{remark}
	\label{rmk:order.in.es}
	The preceding proposition says that ``bigger'' elements in $E(S)$ have ``shorter'' words and ``bigger'' sets (when compared in the same Boolean Algebra). Furthermore, for a fixed $\alpha\in\awstar$ the restricted order in the subset of $E(S)$ given by $\{(\alpha, A, \alpha) \ | \ A\in\acfra\}$ coincides with the natural inclusion order in $\acfra$.
\end{remark}

Let $\xi$ be a filter in $E(S)$ and $p=(\alpha, A,\alpha)$ and $q=(\beta, B, \beta)$ in $\xi$. Since $\xi$ is a filter, then $pq\neq0$; hence, $\alpha$ and $\beta$ are comparable. This says the words of any two elements in a filter are comparable.

\begin{proposition}
	\label{prop:filter.from.finite.word.and.filter}
	Let $\alpha\in\awstar$ and $\ft$ be a filter in $\acfra$. Then
	\[\begin{array}{lll} \xi & = & \displaystyle\bigcup_{A\in\ft}\uset{(\alpha,A,\alpha)} \\
	& = & \{(\alpha_{1,i},A,\alpha_{1,i})\in E(S) \ | \ 0\leq i\leq |\alpha| \ \mbox{and} \ r(A,\alpha_{i+1,|\alpha|})\in\ft\} \end{array}\]
	is a filter in $E(S)$.
\end{proposition}
\begin{proof}
	First we show that the sets in the statement are equal: let $A\in\ft$ and $p=(\beta,B,\beta)\in\uset{(\alpha,A,\alpha)}$. From the order in $E(S)$ and Proposition \ref{prop:order.in.es}, we obtain $\alpha=\beta\alpha'$ and $A\subseteq r(B,\alpha')$ so that $\beta=\alpha_{1,i}$ for some $i$ and $r(B,\alpha_{i+1,|\alpha|})=r(B,\alpha')\in\ft$, since $A\in\ft$ and $\ft$ is a filter.
	
	Conversely, let $p=(\alpha_{1,i},A,\alpha_{1,i})$ be such that $r(A,\alpha_{i+1,|\alpha|})\in\ft$. Together with $(\alpha, r(A,\alpha_{i+1,|\alpha|}),\alpha)\leq (\alpha_{1,i},A,\alpha_{1,i})$, this shows that $p\in\bigcup_{A\in\ft}\uset{(\alpha,A,\alpha)}$.
	
	Now we show that $\xi$ is a filter: clearly $\xi$ is nonempty, $0\notin\xi$ and if $p\in\xi$ and $p\geq q$ then $q\in\xi$. Let $p,q\in\xi$ and choose $A$ and $B$ in $\ft$ such that $p\geq (\alpha,A,\alpha)$ and $q\geq (\alpha,B,\alpha)$. We have that $A\cap B\in\ft$, given that $\ft$ is a filter, whence $r=(\alpha,A\cap B,\alpha)\in\xi$. Since $r\leq p$ and $r\leq q$, the proof is complete.
\end{proof}

\begin{proposition}
	\label{prop:finite.word.and.filter.from.filter}
	Let $\xi$ be a filter in $E(S)$ and suppose there exists a word $\alpha\in\awstar$ such that $(\alpha,A,\alpha)\in\xi$ for some $A\in\acfra$ and $\alpha$ has the largest length among all $\beta$ such that $(\beta,B,\beta)\in\xi$ for some $B\in\acfrg{\beta}$. Define
	\[\ft=\{A\in\acf \ | \ (\alpha,A,\alpha)\in\xi\}.\]
	Then $\ft$ is a filter in $\acfra$ and
	\[\xi = \bigcup_{A\in\ft}\uset{(\alpha,A,\alpha)}.\]
\end{proposition}
\begin{proof}
	When we restrict the order of $E(S)$ to elements with the same word $\alpha$, the order coincides with the order in $\acfra$, hence $\ft$ is a filter. It is easy to see that $\bigcup_{A\in\ft}\uset{(\alpha,A,\alpha)}\subseteq \xi$. On the other hand, let $p=(\beta,B,\beta)\in\xi$ and choose $A\in\acfra$ such that $q=(\alpha,A,\alpha)\in\xi$. Since $\xi$ is a filter, then $p\wedge q=pq\in\xi$ and $pq\neq0$. Therefore, $\beta$ must be a beginning of $\alpha$ (because $\alpha$ has the largest length), say $\alpha=\beta\alpha'$, and $pq=(\alpha,A\cap r(B,\alpha'),\alpha)$. From this we conclude that $A\cap r(B,\alpha')\in\ft$, whence $p\in\uset{(\alpha,A\cap r(B,\alpha'),\alpha)}\subseteq\bigcup_{A\in\ft}\uset{(\alpha,A,\alpha)}$.
\end{proof}

In Proposition \ref{prop:filter.from.finite.word.and.filter}, we construct a filter with a largest word in $E(S)$ from a (finite) word $\alpha$ and a filter in $\acfra$, whereas with Proposition \ref{prop:finite.word.and.filter.from.filter}, we go in the opposite direction. It is easy to see that these constructions are each other's inverses, and therefore there is a bijective correspondence between filters in $E(S)$ with largest word and pairs $(\alpha,\ft)$, where $\alpha$ is a finite word and $\ft$ is a filter in $\acfra$.

Our next goal is to obtain a characterization of the filters in $E(S)$ which do not admit a word with largest length.

\begin{definition}
	\label{def:admissible.and.complete.families}
	Let $\alpha\in\awinf$ and $\{\ftg{F}_n\}_{n\geq0}$ be a family such that $\ftg{F}_n$ is a filter in $\acfrg{\alpha_{1,n}}$ for every $n>0$ and $\ftg{F}_0$ is a filter in $\acf$ or $\ftg{F}_0=\emptyset$. The family $\{\ftg{F}_n\}_{n\geq0}$ is said to be \emph{admissible for} $\alpha$ if
	\[\ftg{F}_n\subseteq\{A\in \acfrg{\alpha_{1,n}} \ | \ r(A,\alpha_{n+1})\in\ftg{F}_{n+1}\}\]
	for all $n\geq0$, and is said to be \emph{complete for} $\alpha$ if
	\[\ftg{F}_n = \{A\in \acfrg{\alpha_{1,n}} \ | \ r(A,\alpha_{n+1})\in\ftg{F}_{n+1}\}\]
	for all $n\geq0$.
\end{definition}

\begin{lemma}
	\label{lemma:equivalence.for.complete.families}
	Let $\alpha\in\awinf$ and $\{\ftg{F}_n\}_{n\geq0}$ be a family such that $\ftg{F}_n$ is a filter in $\acfrg{\alpha_{1,n}}$ for every $n>0$ and $\ftg{F}_0$ is a filter in $\acf$ or $\ftg{F}_0=\emptyset$. Then $\{\ftg{F}_n\}_{n\geq0}$ is complete if and only if
	\[\ftg{F}_n = \{A\in \acfrg{\alpha_{1,n}} \ | \ r(A,\alpha_{n+1,m})\in\ftg{F}_{m}\}\]
	for all $n\geq0$ and all $m>n$.
\end{lemma}
\begin{proof}
	The ``if'' part is clear and the ``only if'' part follows easily by induction using that $r(r(A,\beta),\gamma)=r(A,\beta\gamma)$ if $\beta,\gamma\in\awstar$ are such that $\beta\gamma\in\awstar$.
\end{proof}

\begin{proposition}
	\label{prop:filter.from.admissible.family}
	Let $\alpha\in\awinf$, $\{\ftg{F}_n\}_{n\geq0}$ be an admissible family for $\alpha$ and define
	\[\xi = \bigcup_{n=0}^{\infty}\bigcup_{A\in\ftg{F}_n}\uset{(\alpha_{1,n},A,\alpha_{1,n})}.\]
	Then $\xi$ is a filter in $E(S)$.
\end{proposition}
\begin{proof}
	We already know that, for each $n\geq0$,
	\[\xi_n=\bigcup_{A\in\ftg{F}_n}\uset{(\alpha_{1,n},A,\alpha_{1,n})}\]
	is a filter in $E(S)$ (except if $\ftg{F}_0=\emptyset$, which would then give $\xi_0=\emptyset$). We claim that $\xi_n$ is contained in $\xi_{n+1}$ for every $n\geq0$. Indeed, let $p\in\xi_n$ and choose $q=(\alpha_{1,n},A,\alpha_{1,n})\in\xi_n$ such that $q\leq p$. Since the family is admissible, we conclude that $r(A,\alpha_{n+1})\in\ftg{F}_{n+1}$ and, hence, $r=(\alpha_{1,n+1},r(A,\alpha_{n+1}),\alpha_{1,n+1})\in\xi_{n+1}$. Given that $\xi_{n+1}$ is a filter and $r\leq q\leq p$, then $p\in\xi_{n+1}$. This claim says the filters $\xi_n$ are nested, from where we deduce that their union is a filter.
\end{proof}

\begin{proposition}
	\label{prop:complete.family.from.filter}
	Let $\xi$ be a filter in $E(S)$ and suppose there is no word in $\gamma\in\awstar$ such that $(\gamma,C,\gamma)\in\xi$ for some $C\in\acfrg{\gamma}$ and $\gamma$ possesses largest length among all $\beta$ such that $(\beta,B,\beta)\in\xi$ for some $B\in\acfrg{\beta}$. Then there exists $\alpha\in\awinf$ such that every $p\in\xi$ can be written as $p=(\alpha_{1,n},A,\alpha_{1,n})$ for some $n\geq0$ and some $A\in\acfrg{\alpha_{1,n}}$. Moreover, if we define for each $n\geq0$,
	\[\ftg{F}_n=\{A\in\acf \ | \ (\alpha_{1,n},A,\alpha_{1,n})\in\xi\},\]
	then $\{\ftg{F}_n\}_{n\geq0}$ is a complete family for $\alpha$.
\end{proposition}
\begin{proof}
	For each $n>0$, choose $p=(\beta,B,\beta)\in\xi$ such that $|\beta|\geq n$ and define $\alpha_n=\beta_n$. Since any two elements in $\xi$ have comparable words, then $\alpha_n$ is well defined. Hence, we have a word $\alpha=\alpha_1\alpha_2\ldots\in\awinf$. Let $\ftg{F}_n$ as in the statement. We claim that $\ftg{F}_n$ is nonempty for all $n>0$. Indeed, fix $n>0$ e choose $m\geq n$ such that there is $p=(\alpha_{1,m},A,\alpha_{1,m})\in\xi$. Observe that $q=(\alpha_{1,n},r(\alpha_{1,n}),\alpha_{1,n})\geq p$ and, because $\xi$ is a filter, $q\in\xi$. This shows that $r(\alpha_{1,n})\in\ftg{F}_n$.
	
	Again, since the order in $E(S)$ is the same as the order in $\acfrg{\alpha_{1,n}}$ when we consider only elements with the word $\alpha_{1,n}$, we conclude that $\ftg{F}_n$ is a filter in $\acfrg{\alpha_{1,n}}$ for all $n\geq0$ (except for $n=0$ where $\ftg{F}_0$ may be the empty set).
	
	It remains to show that the family is complete. Fix $n\geq0$ and denote by $\ftg{G}$ the set $\{A\in \acfrg{\alpha_{1,n}} \ | \ r(A,\alpha_{n+1})\in\ftg{F}_{n+1}\}$. Let $A\in\ftg{F}_n$. Since $\ftg{F}_{n+1}$ is nonempty, we can choose $B\in\ftg{F}_{n+1}$. Because $\xi$ is a filter and both $p=(\alpha_{1,n},A,\alpha_{1,n})$ and $q=(\alpha_{1,n+1},B,\alpha_{1,n+1})$ belong to $\xi$, it follows that $p\wedge q = (\alpha_{1,n+1},r(A,\alpha_{n+1})\cap B,\alpha_{1,n+1})\in\xi$. This says $r(A,\alpha_{n+1})\cap B$ is nonempty and belongs to $\ftg{F}_{n+1}$. Since $\ftg{F}_{n+1}$ is a filter in $\acfrg{\alpha_{1,n+1}}$ and $r(A,\alpha_{n+1})\in\acfrg{\alpha_{1,n+1}}$, we have that $r(A,\alpha_{n+1})\in\ftg{F}_{n+1}$ and, therefore, $A\in\ftg{G}$. On the other hand, let $A\in\ftg{G}$ and observe that $p=(\alpha_{1,n+1},r(A,\alpha_{n+1}),\alpha_{1,n+1})\in\xi$. Clearly, $p\leq q=(\alpha_{1,n},A,\alpha_{1,n})$. Again using that $\xi$ is a filter, we obtain $q\in\xi$, which says $A\in\ftg{F}_n$.
\end{proof}

\begin{proposition}
	\label{prop:classification.of.filters.of.infinite.type}
	There is a bijective correspondence between filters in $E(S)$ without largest word and pairs $(\alpha, \{\ftg{F}_n\}_{n\geq0})$, where $\alpha\in\awinf$ and $\{\ftg{F}_n\}_{n\geq0}$ is a complete family for $\alpha$. These correspondences are given by Propositions \ref{prop:filter.from.admissible.family} and \ref{prop:complete.family.from.filter}.
\end{proposition}
\begin{proof}
	Let $\xi$ be a filter in $E(S)$ without a largest word, consider $(\alpha, \{\ftg{F}_n\}_{n\geq0})$ as in Proposition \ref{prop:complete.family.from.filter} and let $\eta$ be the filter given by Proposition \ref{prop:filter.from.admissible.family} from $(\alpha, \{\ftg{F}_n\}_{n\geq0})$; we will show that $\xi=\eta$: clearly $\xi\subseteq\eta$ from the definition of $\eta$. Let $p\in\eta$ and choose $n\geq0$ and $A\in\ftg{F}_n$ such that $(\alpha_{1,n}, A,\alpha_{1,n})\leq p$. By the definition of $\ftg{F}_n$ we have $(\alpha_{1,n}, A,\alpha_{1,n})\in\xi$ and, since $\xi$ is a filter, it follows that $p\in\xi$ and the equality is established.
	
	Conversely, let $\alpha\in\awinf$ and $\{\ftg{F}_n\}_{n\geq0}$ be a complete family for $\alpha$, consider the filter $\xi$ constructed from them as in Proposition \ref{prop:filter.from.admissible.family} and let $(\beta, \{\ftg{G}_n\}_{n\geq0})$ be given from $\xi$ by Proposition \ref{prop:complete.family.from.filter}. We need to show that $(\alpha, \{\ftg{F}_n\}_{n\geq0})=(\beta, \{\ftg{G}_n\}_{n\geq0})$. By construction, we see that $\alpha=\beta$. Fix $m\geq0$ and let $A\in\ftg{F}_m$. Then $(\alpha_{1,m}, A,\alpha_{1,m})\in\xi$, hence $A\in\ftg{G}_m$. 
	
	On the other hand, given $B\in\ftg{G}_m$ the definition of $\ftg{G}_m$ ensures that $(\alpha_{1,m}, B,\alpha_{1,m})\in\xi$; therefore, there exists $k\geq m$ and $C\in\ftg{F}_k$ such that $(\alpha_{1,k}, C,\alpha_{1,k})\leq(\alpha_{1,m}, B,\alpha_{1,m})$. This says that $C\subseteq r(B,\alpha_{m+1,k})$ and, since $\ftg{F}_k$ is a filter in $\acfrg{\alpha_{1,k}}$, it follows that $r(B,\alpha_{m+1,k})\in\ftg{F}_k$. Completeness of $\{\ftg{F}_n\}_{n\geq0}$ and Lemma \ref{lemma:equivalence.for.complete.families} now give $B\in\ftg{F}_m$. Thus, $(\alpha, \{\ftg{F}_n\}_{n\geq0})=(\beta, \{\ftg{G}_n\}_{n\geq0})$ and the proof is complete.
\end{proof}

The previous results classify the filters in $E(S)$ in two types: the ones as in Propositions \ref{prop:filter.from.finite.word.and.filter} and \ref{prop:finite.word.and.filter.from.filter}, which we call \emph{filters of finite type}; and the filters as in Propositions \ref{prop:filter.from.admissible.family}, \ref{prop:complete.family.from.filter} and \ref{prop:classification.of.filters.of.infinite.type}, called \emph{filters of infinite type}.

Next, we present some properties about filters in $E(S)$ that will be useful later; furthermore, we introduce a new notation for the filters.

\begin{proposition}
	\label{prop:properties.about.filters}
	Let $\alpha,\beta\in\awstar$ be such that $\alpha\beta\in\awstar$, $\ft$ be a filter in $\acfra$ and $\ftg{G}$ be a filter in $\acfrg{\alpha\beta}$. The following are equivalent:
	\begin{enumerate}[(i)]
		\item $\{r(A,\beta) \ | \ A\in\ft\}\subseteq\ftg{G}$;
		\item $\ft\subseteq \{A\in\acfra \ | \ r(A,\beta)\in\ftg{G}\}$;
		\item $\displaystyle\bigcup_{A\in\ft}\uset{(\alpha,A,\alpha)}\subseteq \bigcup_{B\in\ftg{G}}\uset{(\alpha\beta,B,\alpha\beta)}$.
	\end{enumerate}
\end{proposition}
\begin{proof}
	The equivalence between (i) and (ii) is obvious. Suppose (iii) is valid and let $A\in\ft$; by (iii), there exists $B\in\ftg{G}$ such that $(\alpha\beta, B, \alpha\beta)\leq(\alpha,A,\alpha)$, hence $B\subseteq r(A,\beta)$. Since $\ftg{G}$ is a filter in $\acfrg{\alpha\beta}$ and $r(A,\beta)\in\acfrg{\alpha\beta}$, therefore $r(A,\beta)\in\ftg{G}$, proving (ii).
	
	Now assume (ii), let $p\in\bigcup_{A\in\ft}\uset{(\alpha,A,\alpha)}$ and choose $A\in\ft$ such that $(\alpha,A,\alpha)\leq p$. We have $(\alpha\beta,r(A,\beta),\alpha\beta)\leq(\alpha,A,\alpha)\leq p$ and, by (ii), $r(A,\beta)\in\ftg{G}$, so that $p\in\bigcup_{B\in\ftg{G}}\uset{(\alpha\beta,B,\alpha\beta)}$ which gives (iii).
\end{proof}

\begin{proposition}
	\label{prop:completion.of.admissible.families}
	Let $\alpha\in\awstar$ and $\{\ftg{F}_n\}_{n\geq0}$ be  an admissible family for $\alpha$. For each $n\geq0$, define
	\[\overline{\ft}_n=\{A\in\acfrg{\alpha_{1,n}} \ | \ r(A,\alpha_{n+1,m})\in\ftg{F}_m \ \mbox{for some} \ m\geq n\}.\]
	Then:
	\begin{enumerate}[(i)]
		\item $\ftg{F}_n\subseteq \overline{\ft}_n$, for all $n\geq0$;
		\item $\{\overline{\ft}_n\}_{n\geq0}$ is a complete family for $\alpha$;
		\item $(\alpha, \{\ftg{F}_n\}_{n\geq0})$ and $(\alpha, \{\overline{\ft}_n\}_{n\geq0})$ generate the same filter in $E(S)$ (see Proposition \ref{prop:filter.from.admissible.family}).
		\item If $\xi$ is the filter associated with $(\alpha, \{\ftg{F}_n\}_{n\geq0})$, then the pair associated with $\xi$ (by Proposition \ref{prop:complete.family.from.filter}) is $(\alpha, \{\overline{\ft}_n\}_{n\geq0})$.
	\end{enumerate}	
\end{proposition}
\begin{proof}
	Item (i) is immediate from the definition of $\overline{\ft}_n$. We claim $\overline{\ft}_n$ is a filter in $\acfrg{\alpha_{1,n}}$ (or possibly the empty set when $n=0$): clearly, $\emptyset\notin\overline{\ft}_n$. Let $A_1$ and $A_2$ be in $\overline{\ft}_n$ and choose $m_1,m_2\geq n$ such that $r(A_1,\alpha_{n+1,m_1})\in\ftg{F}_{m_1}$ and $r(A_2,\alpha_{n+1,m_2})\in\ftg{F}_{m_2}$. Define $m=\max\{m_1,m_2\}$ and note that $r(A_1,\alpha_{n+1,m})$ and $r(A_2,\alpha_{n+1,m})$ both belong to $\ftg{F}_m$, given that $\{\ftg{F}_n\}_{n\geq0}$ is admissible. Because $\ftg{F}_m$ is a filter in $\acfrg{\alpha_{1,m}}$ and the labelled space is weakly left resolving, we have that
	\[r(A_1\cap A_2,\alpha_{n+1,m})=r(A_1,\alpha_{n+1,m})\cap r(A_2,\alpha_{n+1,m})\in\ftg{F}_m\]
	and therefore $A_1\cap A_2\in\overline{\ft}_n$. Now, let $A\in\overline{\ft}_n$ and $B\in\acfrg{\alpha_{1,n}}$ be such that $B\supseteq A$. Choose $m\geq n$ such that $r(A,\alpha_{n+1,m})\in\ftg{F}_m$ and observe that $r(B,\alpha_{n+1,m})\supseteq r(A,\alpha_{n+1,m})$. Since $\ftg{F}_m$ is a filter, we have that $r(B,\alpha_{n+1,m})\in\ftg{F}_m$, which says $B\in\overline{\ft}_n$. This completes the proof of our claim.
	
	Next we show that $\{\overline{\ft}_n\}_{n\geq0}$ is complete: Let
	\[\ftg{G}_n=\{A\in\acfrg{\alpha_{1,n}} \ | \ r(A,\alpha_{n+1})\in\overline{\ft}_{n+1}\}.\]
	We need to show that $\overline{\ft}_n=\ftg{G}_n$. To see this, let $A\in\overline{\ft}_n$ and choose $m\geq n$ such that $r(A,\alpha_{n+1,m})\in\ftg{F}_m$. If $m=n$ then $A\in\ftg{F}_n$ and, since $\{\ftg{F}_n\}_{n\geq0}$ is admissible, we have $r(A,\alpha_{n+1})\in\ftg{F}_{n+1}\subseteq \overline{\ft}_{n+1}$, which says $A\in\ftg{G}_n$. If $m>n$, then $r(r(A,\alpha_{n+1}),\alpha_{n+2,m})=r(A,\alpha_{n+1,m})\in\ftg{F}_m$ and therefore $r(A,\alpha_{n+1})\in\overline{\ft}_{n+1}$, whence $A\in\ftg{G}_n$. On the other hand, let $B\in\ftg{G}_n$. Then $r(B,\alpha_{n+1})\in\overline{\ft}_{n+1}$ and we can take $m\geq n+1$ such that $r(B,\alpha_{n+1,m})=r(r(B,\alpha_{n+1}),\alpha_{n+2,m})\in\ftg{F}_m$. It follows that $B\in\overline{\ft}_n$ and the proof of (ii) is complete.
	
	In order to show (iii), denote by $\xi$ and $\bar{\xi}$ the filters associated with $(\alpha, \{\ftg{F}_n\}_{n\geq0})$ and $(\alpha, \{\overline{\ft}_n\}_{n\geq0})$, respectively. Clearly, $\xi\subseteq\bar{\xi}$; let $p\in\bar{\xi}$ and choose $n\geq0$ and $A\in\overline{\ft}_n$ such that $(\alpha_{1,n},A,\alpha_{1,n})\leq p$. Take $m\geq n$ such that $r(A,\alpha_{n+1,m})\in\ftg{F}_m$ and observe that $q=(\alpha_{1,m},r(A,\alpha_{n+1,m}),\alpha_{1,m})\in\xi$ and $q\leq p$. $\xi$ is a filter, therefore $p\in\xi$, which shows (iii).
	
	Item (iv) is consequence from (iii) and Proposition \ref{prop:classification.of.filters.of.infinite.type}.
\end{proof}

\begin{proposition}
	\label{prop:properties.about.filters.of.finite.type}
	Let $\xi$ be a filter in $E(S)$ of finite type and $(\alpha,\ft)$ be its associated pair. For each $n\in\{0,1,\ldots,|\alpha|\}$, define
	\[\ftg{F}_{n}= \{A\in\acf \ | \ (\alpha_{1,n},A,\alpha_{1,n})\in\xi\}.\]
	Then:
	\begin{enumerate}[(i)]
		\item $\ftg{F}_{n}$ is a filter in $\acfrg{\alpha_{1,n}}$ for every $n\in\{1,\ldots,|\alpha|\}$ and $\ftg{F}_0$ is either empty or a filter in $\acf$;
		\item $\ftg{F}_n = \{A\in \acfrg{\alpha_{1,n}} \ | \ r(A,\alpha_{n+1,m})\in\ftg{F}_{m}\}$, for all $0\leq n<m\leq|\alpha|$.
	\end{enumerate}
\end{proposition}
\begin{proof}
	These results are proved exactly in the same way as Proposition \ref{prop:complete.family.from.filter} and Lemma \ref{lemma:equivalence.for.complete.families}.
\end{proof}

We summarize the results obtained until now. There are two kinds of filters in $E(S)$: of infinite and finite type. The former is characterized by an infinite word $\alpha$ and a complete family $\{\ftg{F}_n\}_{n\geq0}$ for $\alpha$. Clearly, for $n\leq m$, $\ftg{F}_n$ is completely determined from $\ftg{F}_m$. When we need to create a filter of infinite type, we need only an infinite word $\alpha$ and an admissible family for $\alpha$. Indeed, an admissible family can be completed to a complete family by using Proposition \ref{prop:completion.of.admissible.families}. Sometimes, we use Proposition \ref{prop:properties.about.filters} to obtain equivalent formulations for the admissibility condition in Definition \ref{def:admissible.and.complete.families}.

The filters of finite type are determined by a finite word $\alpha$ and a filter $\ft$ in $\acfra$. Aiming to obtain a standard way to work with all filters, we now describe filters of finite type in a slightly different form: analogous to Definition \ref{def:admissible.and.complete.families}, to each finite word $\alpha$ we can define admissible or complete families $\{\ftg{F}_n\}_{0\leq n\leq|\alpha|}$ for $\alpha$ (when $\alpha$ is the empty word $\eword$, we do not allow $\ftg{F}_0=\emptyset$). Due to Proposition \ref{prop:completion.of.admissible.families}, admissible families can be completed and, since the filter $\ftg{F}_{|\alpha|}$ determines the family (Proposition \ref{prop:properties.about.filters.of.finite.type}), then it is clear that there is a correspondence between filters of finite type and pairs $(\alpha, \{\ftg{F}_n\}_{0\leq n\leq|\alpha|})$, where $\alpha$ is a finite word and $\{\ftg{F}_n\}_{0\leq n\leq|\alpha|}$ is a complete family for $\alpha$.

Combining the results obtained in this section, we obtain the following description for the filters in $E(S)$.

\begin{theorem}
	Let $\lspace$ be a labelled space which is weakly left resolving, and let $S$ be its associated inverse semigroup. Then there is a bijective correspondence between filters in $E(S)$ and pairs $(\alpha, \{\ftg{F}_n\}_{0\leq n\leq|\alpha|})$, where $\alpha\in\awleinf$ and $\{\ftg{F}_n\}_{0\leq n\leq|\alpha|}$ is a complete family for $\alpha$ (understanding that ${0\leq n\leq|\alpha|}$ means $0\leq n<\infty$ when $\alpha\in\awinf$).
\end{theorem}

In view of this, we will use the following notation: when a filter $\xi$ in $E(S)$ with associated word $\alpha\in\awleinf$ is given, we write $\xi=\xia$ to stress the word $\alpha$. Furthermore, we denote by $\xi^{\alpha}_n$ (or $\xi_n$ when there is no ambiguity) the filter $\ftg{F}_n$ in $\acfrg{\alpha_{1,n}}$ of the complete family associated with $\xia$ (obviously, $n\leq|\alpha|$). To be specific:
\[\xi^{\alpha}_n=\{A\in\acf \ | \ (\alpha_{1,n},A,\alpha_{1,n}) \in \xia\}\]
and the family $\{\xi^{\alpha}_n\}_{0\leq n\leq|\alpha|}$ satisfies
\[\xi^{\alpha}_n = \{A\in \acfrg{\alpha_{1,n}} \ | \ r(A,\alpha_{n+1,m})\in\xi^{\alpha}_{m}\}\]
for all $0\leq n<m\leq|\alpha|$.

Finally we remark that, in order to create a filter of finite type, we need only a word $\alpha\in\awstar$ and a filter in $\acfra$, whereas to create a filter of infinite type, it is enough to have a word $\alpha\in\awinf$ and an admissible family for $\alpha$.

For future reference, we denote by $\filt$ the set of all filters in $E(S)$ and by $\filtw{\alpha}$ the subset of $\filt$ of those filters whose associated word is $\alpha\in\awleinf$.

\section{Ultrafilters in $E(S)$}
\label{section:ultrafilters.in.es}

Next we describe the ultrafilters in $E(S)$. Since the notion of ultrafilter is associated with inclusion among filters, we enunciate an easy, but useful, proposition which characterizes the order (given by inclusion) in $\filt$.

\begin{proposition}
	\label{prop:order.in.f}
	Let $\xia,\eta^{\beta}\in\filt$. Then $\xia\subseteq\eta^{\beta}$ if and only if $\alpha$ is a beginning of (or equal to) $\beta$ and $\xi_n\subseteq\eta_n$ for all $0\leq n\leq|\alpha|$.
\end{proposition}
\begin{proof}
	Trivial.
\end{proof}

\begin{remark}
	\label{rmk:order.in.f}
	The proposition above can be simplified in the case $\alpha$ is a finite word: the statement would then be $\xia\subseteq\eta^{\beta}$ if and only if $\alpha$ is a beginning of (or equal to) $\beta$ and $\xi_{|\alpha|}\subseteq\eta_{|\alpha|}$.
\end{remark}

\begin{proposition}
	\label{prop:ultrafilters.of.infinite.type}
	Let $\xia\in\filt$ be a filter of infinite type, that is, $\alpha\in\awinf$. Then $\xia$ is an ultrafilter if and only if $\{\xi_n\}_{n\geq 0}$ is maximal among all complete families for $\alpha$, where maximal means that if $\{\ftg{F}_n\}_{n\geq 0}$ is complete family for $\alpha$ such that $\xi_n\subseteq\ftg{F}_n$ for all $n\geq0$, then $\xi_n=\ftg{F}_n$ for all $n\geq0$.
\end{proposition}
\begin{proof}
	First suppose $\xia$ is an ultrafilter and let $\{\ftg{F}_n\}_{n\geq0}$ be a complete family for $\alpha$ such that $\xi_n\subseteq\ftg{F}_n$ for all $n\geq0$. If $\eta$ is the filter associated with $(\alpha, \{\ftg{F}_n\}_{n\geq0})$ then $\xia\subseteq\eta$, by the previous result. Since $\xia$ is an ultrafilter, therefore $\xia=\eta$, hence $\xi_n=\ftg{F}_n$ for all $n\geq0$, that is, $\{\xi_n\}_{n\geq 0}$ is maximal.
	
	On the other hand, suppose $\{\xi_n\}_{n\geq 0}$ is maximal and let $\eta^{\beta}\in\filt$ be such that $\xia\subseteq\eta^{\beta}$. By the previous result we must have $\beta=\alpha$ and $\xi_n\subseteq\eta_n$ for all $n\geq0$; since $\{\xi_n\}_{n\geq 0}$ is maximal, we have that  $\xi_n=\eta_n$ for all $n$ and therefore $\xia=\eta^{\beta}$, showing that $\xia$ is an ultrafilter.
\end{proof}

\begin{proposition}
	\label{prop:ultrafilters.of.finite.type}
	Let $\xia\in\filt$ be a filter of finite type, that is, $\alpha\in\awstar$. Then $\xia$ is an ultrafilter if and only if $\xi_{|\alpha|}$ is an ultrafilter in $\acfra$ and for each letter $b\in\alf$, there exists $A\in\xi_{|\alpha|}$ such that $r(A,b)=\emptyset$.
\end{proposition}
\begin{proof}
	$(\Rightarrow)$ Suppose $\xia$ is an ultrafilter, and let $\ft$ be a filter in $\acfra$ such that $\xi_{|\alpha|}\subseteq\ft$. The filter $\eta$ associated with $(\alpha,\ft)$ contains $\xia$, therefore $\eta=\xia$ and so $\ft=\xi_{|\alpha|}$, whence $\xi_{|\alpha|}$ is an ultrafilter.
	
	Now, suppose by contradiction there exists $b\in\alf$ such that $r(A,b)\neq\emptyset$ for all $A\in\xi_{|\alpha|}$. Consider the subset of $\acfrg{\alpha b}$ given by
	\[Y=\{r(A,b) \ | \ A\in\xi_{|\alpha|}\}.\]
	
	Given that $\xi_{|\alpha|}$ is a filter and the labelled space is weakly left resolving, we have that $Y$ is closed under intersections and, by hypothesis, $\emptyset\notin Y$; it follows from Proposition \ref{prop:subsets.that.generate.filters} that $\usetr{Y}{\acfrg{\alpha b}}$ is a filter in $\acfrg{\alpha b}$. Consider then the filter $\eta$ associated with the pair $(\alpha b,\usetr{Y}{\acfrg{\alpha b}})$. Since $\eta_{|\alpha|} = \{A\in \acfra \ | \ r(A,b)\in\usetr{Y}{\acfrg{\alpha b}}\}\supseteq \xi_{|\alpha|}$, we would then have $\xia\subsetneqq\eta$ by Remark \ref{rmk:order.in.f}, contradicting the fact that $\xia$ is an ultrafilter.
	
	$(\Leftarrow)$ Suppose $\xia$ is not an ultrafilter, so that there is a $\eta^{\beta}\in\filt$ such that $\xia\subsetneqq\eta^{\beta}$, noting that $\alpha$ is a beginning of $\beta$, by Proposition \ref{prop:order.in.f}. If $\beta=\alpha$ then $\xi_{|\alpha|}\subsetneqq\eta_{|\alpha|}$ (by Remark \ref{rmk:order.in.f}), hence $\xi_{|\alpha|}$ is not an ultrafilter in $\acfra$. If $\beta\neq\alpha$, then $\beta=\alpha\gamma$ with $\gamma\neq\eword$. Let $b$ be the first letter of $\gamma$ and observe that $\xi_{|\alpha|}\subseteq\eta_{|\alpha|}=\{A\in\acfra \ | \ r(A,b)\in\eta_{|\alpha|+1}\}$; it follows from Proposition \ref{prop:properties.about.filters} that $\{r(A,b) \ | \ A\in\xi_{|\alpha|}\}\subseteq\eta_{|\alpha|+1}$ and, because $\eta_{|\alpha|+1}$ is a filter, we have $r(A,b)\neq\emptyset$ for all $A\in\xi_{|\alpha|}$, a contradiction.
\end{proof}

\begin{remark}
	\label{rmk:ultrafilters.of.finite.type}
	The condition ``for each letter $b\in\alf$, there exists $A\in\xi_{|\alpha|}$ such that $r(A,b)=\emptyset$'' in the previous result has other variations; indeed, if $\xia\in\filt$ is a filter of finite type, it is then easy to see that the following conditions are equivalent:
	\begin{enumerate}[(i)]
		\item For each letter $b\in\alf$, there exists $A\in\xi_{|\alpha|}$ such that $r(A,b)=\emptyset$;
		\item For each letter $b\in\alf$ such that $\alpha b\in\awstar$, there exists $A\in\xi_{|\alpha|}$ such that $r(A,b)=\emptyset$;
		\item For each word $\beta\in\awstar$, there exists $A\in\xi_{|\alpha|}$ such that $r(A,\beta)=\emptyset$;
		\item For each word $\beta\in\awstar$ such that $\alpha\beta\in\awstar$, there exists $A\in\xi_{|\alpha|}$ such that $r(A,\beta)=\emptyset$.
	\end{enumerate}
\end{remark}

The previous two results characterize the ultrafilters in $E(S)$, but natural questions related to Propositions \ref{prop:ultrafilters.of.infinite.type} and \ref{prop:ultrafilters.of.finite.type} can be raised; for instance, if $\alpha\in\awinf$ and $\xia$ is a filter, it is clear that if $\xi_n$ is ultrafilter for every $n\geq0$, then $\xia$ is an ultrafilter. The converse, however, is not true in general, as the next example illustrates.

\begin{example}
	\label{example:maximal.family.with.a.nonultrafilter}
	Consider the labelled graph $\lgraph$ given by

	
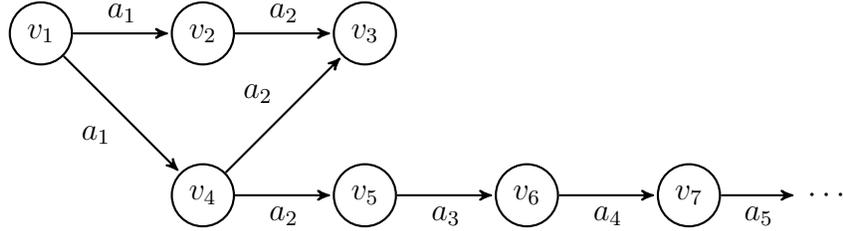
\begin{tikzpicture}[->,>=stealth',shorten >=1pt,auto,thick,main node/.style={circle,draw}]
	\node[main node] (1) {$v_1$};
	\node[main node] (2) [right=1.3cm of 1] {$v_2$};
	\node[main node] (3) [right=1.3cm of 2] {$v_3$};
	\node[main node] (4) [below=1.3cm of 2] {$v_4$};
	\node[main node] (5) [right=1.3cm of 4] {$v_5$}; 
	\node[main node] (6) [right=1.3cm of 5] {$v_6$}; 
	\node[main node] (7) [right=1.3cm of 6] {$v_7$}; 
	\node (8) [right=1cm of 7] {$\cdots$};
	
	\draw[->] (1) to node [above] {${a_1}$} (2);
	\draw[->] (2) to node [above] {${a_2}$} (3);
	\draw[->] (1) to node [below left] {${a_1}$} (4);
	\draw[->] (4) to node [above left] {${a_2}$} (3);
	\draw[->] (4) to node [below] {${a_2}$} (5);
	\draw[->] (5) to node [below] {${a_3}$} (6);
	\draw[->] (6) to node [below] {${a_4}$} (7);
	\draw[->] (7) to node [below] {${a_5}$} (8);
	\end{tikzpicture}
	
	and define $\acf$ as the family of all subsets $A\subseteq\dgraph^0$ such that $v_2\in A$ whenever $v_4\in A$ and $v_3\in A$ whenever $v_5\in A$. It is not difficult to check that $\acf$ is a weakly left resolving accommodating family for $\lgraph$.

	Consider the labelled path $\alpha = a_1a_2a_3\cdots$. We claim there exists a unique complete family $\{\ftg{F}_n\}_{n\geq0}$ for $\alpha$ (which is therefore maximal), given by $\ft_0 = \{\{v_1\}\}$, $\ft_1 = \{\{v_2,v_4\}\}$, $\ft_2=\{\{v_3,v_5\}\}$ and $\ft_n=\{\{v_{n+3}\}\}$ for $n\geq 3$: to prove this, first observe that $\acf_{\alpha_{1,1}}=\{\emptyset,\{v_2\},\{v_2,v_4\}\}$, $\acf_{\alpha_{1,2}}=\{\emptyset,\{v_3\},\{v_3,v_5\}\}$ and $\acf_{\alpha_{1,n}}=\{\emptyset,\{v_{n+3}\}\}$ for $n\geq 3$. It is easily checked that the given family is complete for $\alpha$, and we see that for $n\geq 3$ the only possible filter in $\acf_{\alpha_{1,n}}$ is $\{\{v_{n+3}\}\}$. Since in a complete family each filter determines uniquely the ones that come before it, this establishes the desired uniqueness (for instance, $\ft_2=\{A\in\acf_{\alpha_{1,2}}:r(A,a_3)\in\ft_3\}$ and, since $r(\{v_3\},a_3)=\emptyset$, then $\ft_2=\{\{v_3,v_5\}\}$. Similarly, the only possibility for $\ft_1$ is $\{\{v_2,v_4\}\}$).	
	
	By Proposition \ref{prop:ultrafilters.of.infinite.type}, we obtain an ultrafilter $\xi^\alpha$ on $E(S)$ such that $\xi_n=\ft_n$ for all $n$; however, $\xi_2=\{\{v_3,v_5\}\}$ is not an ultrafilter on $\acf_{\alpha_{1,2}}$ because it is contained in the filter $\{\{v_3\},\{v_3,v_5\}\}$.
\end{example}

On the other hand, if we assume that the family $\acf$ is closed under relative complements, that is, for every $A,B\in\acf$ one has $A\setminus B\in\acf$, then the situation described in the previous example does not occur and, in view of this, we assume from now on that the accommodating family $\acf$ of the labelled space $\lspace$ is closed under relative complements; this ensures in particular that $\acf=\acfrg{\eword}$ is a generalized Boolean algebra, $\acfrg{\alpha}$ is a Boolean algebra for $\alpha\in\awplus$ and $r(A\setminus B, \beta) = r(A,\beta)\setminus r(B,\beta)$ for $\beta\in\awstar$, see Remark \ref{rmk:b.alpha.boolean.algebra} and Lemma \ref{lemma:wlr.closedcomp.relrangeisnice}. We recall this hypothesis was already used by Bates, Pask and Willis in \cite{2011arXiv1106.1484B}, and Bates, Carlsen and Pask in \cite{2012arXiv1203.3072B}.

\begin{proposition}
	\label{prop:family.ultrafilters}
	Suppose the accommodating family $\acf$ is closed under relative complements and let $\xia$ be a filter in $E(S)$. If $\xi_n$ is an ultrafilter for $n$ with $0<n\leq|\alpha|$, then $\xi_m$ is an ultrafilter for every $0<m<n$. If furthermore $\xi_0$ is nonempty, then $\xi_0$ is also an ultrafilter.
\end{proposition}
\begin{proof}
	Fix $0<m<n$ and consider the relative range function,  $\newf{r(\,\cdot\,,\alpha_{m+1,n})}{\acfrg{\alpha_{1,m}}}{\acfrg{\alpha_{1,n}}}$. The fact that $\acfrg{\alpha_{1,m}}$ and $\acfrg{\alpha_{1,n}}$ are Boolean algebras and $r(\,\cdot\,,\alpha_{m+1,n})$ is a morphism of Boolean algebras ensures that the inverse image of an ultrafilter in $\acfrg{\alpha_{1,n}}$ under $r(\,\cdot\,,\alpha_{m+1,n})$ is an ultrafilter in $\acfrg{\alpha_{1,m}}$. By the definition of $\xi_m$ and $\xi_n$, we see that $\xi_m$ is the inverse image of $\xi_n$; thus, if $\xi_n$ is an ultrafilter, then $\xi_m$ is an ultrafilter.
	
	Now, suppose that $\xi_n$ is an ultrafilter and $\xi_0\neq\emptyset$. We will show that $\xi_0$ satisfies Proposition \ref{prop:ultrafilter.intersection} and will use Proposition \ref{prop:ultrafilter.in.boolean.algebra} to characterize $\xi_n$. Let $C\in\acf$ be such that $C\cap B\neq\emptyset$ for all $B\in\xi_0$ and fix any $A\in\xi_0$. By the condition of $C$, $A\setminus C$ does not belong to $\xi_0$. This says that $r(A,\alpha_{1,n})\in\xi_n$ and $r(A\setminus C,\alpha_{1,n})\notin\xi_n$. Thus, using the notation $D^c=r(\alpha_{1,n})\setminus D$,
	\[\begin{array}{lcl}
	r(A\setminus C,\alpha_{1,n})\notin\xi_n & \Rightarrow & r(A,\alpha_{1,n})\setminus r(C,\alpha_{1,n})\notin\xi_n \\
	& \Rightarrow & r(A,\alpha_{1,n})\cap r(C,\alpha_{1,n})^c\notin\xi_n \\
	& \Rightarrow & r(A,\alpha_{1,n})^c\cup r(C,\alpha_{1,n})\in\xi_n \ \ \ \ \ \ \ \ \ \ \ \mbox{(by \ref{prop:ultrafilter.in.boolean.algebra})} \\
	& \Rightarrow & r(A,\alpha_{1,n})\cap(r(A,\alpha_{1,n})^c\cup r(C,\alpha_{1,n}))\in\xi_n \\ 
	& \Rightarrow & r(A,\alpha_{1,n})\cap r(C,\alpha_{1,n})\in\xi_n \\
	& \Rightarrow & r(C,\alpha_{1,n})\in\xi_n \\
	& \Rightarrow & C\in\xi_0. \end{array}\]
	This shows that $\xi_0$ is an ultrafilter.
\end{proof}

\begin{remark}
	\label{rmk:ultrafilter.of.finite.type}
	Let $\xi^{\alpha}$ be an ultrafilter as in Proposition \ref{prop:ultrafilters.of.finite.type}. If the accommodating family $\acf$ is closed under relative complements, then the previous proposition applies and therefore $\xi_n$ is an ultrafilter for all $0\leq n\leq|\alpha|$ (except, possibly, for $n=0$).
\end{remark}

\begin{proposition}
	\label{prop:ultrafilters.of.infinite.type.2}
	Suppose the accommodating family $\acf$ is closed under relative complements and let $\xia$ be a filter of infinite type. Then $\xia$ is an ultrafilter if and on if $\xi_n$ is an ultrafilter for every $n>0$ and $\xi_0$ is either an ultrafilter or the empty set.
\end{proposition}
\begin{proof}
	The ``if'' part is obvious, so assume $\xia$ is an ultrafilter and fix $n>0$. Let $A\in\acfrg{\alpha_{1,n}}$ be such that $A^c=r(\alpha_{1,n}) \setminus A$ does not belong to $\xi_n$. By using the characterization given by Proposition \ref{prop:ultrafilter.in.boolean.algebra} for ultrafilters in Boolean algebras, it suffices to verify that $A\in\xi_n$ to show that $\xi_n$ is an ultrafilter.
	
	Since $p=(\alpha_{1,n},A^c,\alpha_{1,n})\notin\xia$ and $\xia$ is an ultrafilter, it follows from Proposition \ref{prop:ultrafilter.intersection} that there exists $q=(\alpha_{1,m},B,\alpha_{1,m})\in\xia$ such that $pq=0$. There are three cases to consider: if $m=n$, then $B\cap A^c=\emptyset$, which would give $B\subseteq A$ and thus $A\in\xi_n$, since $\xi_n$ is a filter. The case $m<n$ is analogous using $\tilde{q}=(\alpha_{1,n},r(B,\alpha_{m+1,n}),\alpha_{1,n})$ instead of $q$, and the case $m>n$ follows by using $\tilde{p}=(\alpha_{1,m},r(A^c,\alpha_{n+1,m}),\alpha_{1,m})$ instead of $p$ to obtain that $r(A,\alpha_{n+1,m})\in\xi_m$, and therefore $A\in\xi_n$. This shows that $\xi_n$ is an ultrafilter for $n>0$.
	
	The case $n=0$ follows from Proposition \ref{prop:family.ultrafilters}.
\end{proof}

The results of this section therefore give the following description of the ultrafilters in $E(S)$.

\begin{theorem}\label{thm:ultrafilters}
	Let $\lspace$ be a labelled space which is weakly left resolving, and let $S$ be its associated inverse semigroup. Then the ultrafilters in $E(S)$ are:
	\begin{enumerate}[(i)]
		\item The filters of finite type $\xia$ such that $\xi_{|\alpha|}$ is an ultrafilter in $\acfra$ and for each  $b\in\alf$ there exists $A\in\xi_{|\alpha|}$ such that $r(A,b)=\emptyset$.
		\item The filters of infinite type $\xia$ such that the family $\{\xi_n\}_{n\geq 0}$ is maximal among all complete families for $\alpha$.
	\end{enumerate}
	Suppose in addition that the accommodating family $\acf$ is closed under relative complements. Then (ii) can be replaced with
	\begin{enumerate}[(i)']
		\setcounter{enumi}{1}
		\item The filters of infinite type $\xia$ such that $\xi_n$ is an ultrafilter for every $n>0$ and $\xi_0$ is either an ultrafilter or the empty set.
	\end{enumerate}
\end{theorem}

\begin{remark}
	\label{rmk.ultrafilters.of.infinite.type}
	When the accommodating family $\acf$ is closed under relative complements, we can use Proposition \ref{prop:family.ultrafilters} to conclude that a filter $\xia$ of infinite type is an ultrafilter if and only for every $n\in\mathbb{N}$ there exists $m\geq n$ such that $\xi_m$ is an ultrafilter.
\end{remark}

\section{Tight Filters in $E(S)$}
\label{section.tight.filters.in.es}

Let $\ftight$ be the set of all tight filters in $E(S)$. If we give to it the topology induced from the topology of pointwise convergence of characters, via the bijection between tight characters and tight filters given after Definition \ref{def:character.tight}, it follows that $\ftight$ is (homeomorphic to) the tight spectrum of $E(S)$.

Now we characterize the tight filters in $E(S)$, a goal that is accomplished with Theorem \ref{thm:tight.filters.in.es}. From now on the accommodating family $\acf$ is required to be closed under relative complements, as above. We begin with some auxiliary results.

\begin{proposition}
	\label{prop:tight.filter.is.ultrafilter.in.b.alpha}
	Suppose the accommodating family $\acf$ is closed under relative complements and let $\xia$ be a tight filter in $E(S)$. Then, for every $0\leq n\leq|\alpha|$, $\xi_n$ is an ultrafilter in $\acfrg{\alpha_{1,n}}$ ($\xi_0$ may also be the empty set).
\end{proposition}
\begin{proof}
	Fix $n>0$ and suppose $\xi_n$ is not an ultrafilter. There exists then $A\in\acfrg{\alpha_{1,n}}$ such that $A$ and $A^c=r(\alpha_{1,n}) \setminus A$ does not belong to $\xi_n$.
	
	Let $p=(\alpha_{1,n},r(\alpha_{1,n}),\alpha_{1,n})$, $Z=\{(\alpha_{1,n},A,\alpha_{1,n}), (\alpha_{1,n},A^c,\alpha_{1,n})\}$ and note that $p\in\xia$ and $Z$ is a cover for $p$; indeed, for every $q\leq p$, we must have $q=(\alpha_{1,n}\beta,B,\alpha_{1,n}\beta)$ for some $\beta\in\awstar$ and some $B\in\acfrg{\alpha_{1,n}\beta}$. Since $B\subseteq r(\alpha_{1,n}\beta) = r(A,\beta)\cup r(A^c,\beta)$, then $B\cap r(A,\beta)\neq\emptyset$ or $B\cap r(A^c,\beta)\neq\emptyset$ and, therefore, $q\cdot(\alpha_{1,n},A,\alpha_{1,n})\neq0$ or $q\cdot(\alpha_{1,n},A^c,\alpha_{1,n})\neq0$. This shows $Z$ is a cover for $p$. Since $Z\cap\xia=\emptyset$, we see that $\xia$ is not tight (see Proposition \ref{prop:tight.filter.characterization}), a contradiction.
	
	Now suppose $\xi_0$ is nonempty and not an ultrafilter. By Proposition \ref{prop:ultrafilter.intersection}, there is $A\in\acf$ such that $A\notin\xi_0$ and $A\cap B\neq\emptyset$ for all $B\in \xi_0$. Fix any $B\in\xi_0$ and define $p=(\eword,B,\eword)$ and \[Z=\{(\eword,B\cap A,\eword),(\eword,B\setminus A,\eword)\}.\] An argument similar to the one used above shows that $Z$ is a cover for $p$. Since $p\in\xia$ and $Z\cap\xia=\emptyset$ (because $B\cap A\subseteq A\notin\xi_0$ and $A\cap(B\setminus A)=\emptyset$), it follows that $\xia$ is not tight, again a contradiction and the result is proved.
\end{proof}

\begin{corollary}
	\label{corollary:tight.filters.of.infinite.type}
	Suppose the accommodating family $\acf$ is closed under relative complements and let $\xi$ be a filter of infinite type. Then $\xi$ is tight if and only if $\xi$ is an ultrafilter.
\end{corollary}
\begin{proof}
	This follows from Propositions \ref{prop:tight.filter.is.ultrafilter.in.b.alpha}, \ref{prop:ultrafilter.is.tight} and Theorem \ref{thm:ultrafilters}.
\end{proof}

\begin{corollary}
	\label{corollary:nontight.filters.of.finite.type}
	Suppose the accommodating family $\acf$ is closed under relative complements and let $\xia$ be a filter of finite type. If $\xi_{|\alpha|}$ is not an ultrafilter, then $\xia$ is not tight.
\end{corollary}
\begin{proof}
	Obvious from Proposition \ref{prop:tight.filter.is.ultrafilter.in.b.alpha}.
\end{proof}

Before stating our main theorem,  for each $\alpha\in\awstar$ recall from Remark \ref{rmk:b.alpha.boolean.algebra} that 
\[X_{\alpha}=\{\ft\subseteq\acfra \ | \ \ft \ \mbox{is an ultrafilter in} \ \acfra\}\]
and also define
\[X_{\alpha}^{sink}=\{\ft\in X_{\alpha} \ | \ \forall\,b\in\alf, \ \exists\,A\in\ft \ \mbox{such that} \ r(A,b)=\emptyset\}.\]
Observe that $X_{\alpha}^{sink}$ can be defined in many equivalent ways, using Remark \ref{rmk:ultrafilters.of.finite.type}.

Suppose the accommodating family $\acf$ to be closed under relative complements. In this case, for every $\alpha,\beta\in\awplus$ such that $\alpha\beta\in\awplus$, the map $\newf{r(\,\cdot\,,\beta)}{\acfrg{\alpha}}{\acfrg{\alpha\beta}}$ is a morphism of Boolean algebras and, therefore, we have its dual morphism $\newf{f_{\alpha[\beta]}}{X_{\alpha\beta}}{X_{\alpha}}$ given by $f_{\alpha[\beta]}(\ft)=\{A\in\acfra \ | \ r(A,\beta)\in\ft\}$ (this is just the inverse image of $\ft$ under  $r(\,\cdot\,,\beta)$), as in the proof of Proposition \ref{prop:family.ultrafilters}. When $\alpha=\eword$,  if $\ft\in\acfrg{\beta}$ then $\{A\in\acf \ | \ r(A,\beta)\in\ft\}$ is either an ultrafilter in $\acf=\acfrg{\eword}$ or the empty set, and we can therefore consider $\newf{f_{\eword[\beta]}}{X_{\beta}}{X_{\eword}\cup\{\emptyset\}}$.

Employing this new notation, if $\xia$ is a filter in $E(S)$ and $0\leq m<n\leq|\alpha|$, then $\xi_m=f_{\alpha_{1,m}[\alpha_{m+1,n}]}(\xi_n)$.

If we endow the sets $X_{\alpha}$ with the topology given by the convergence of filters stated at the end of Section \ref{subsection:filters.and.characters} (this is the pointwise convergence of characters), it is clear that the functions $f_{\alpha[\beta]}$ are continuous. Furthermore, it is easy to see that $f_{\alpha[\beta\gamma]}=f_{\alpha[\beta]}\circ f_{\alpha\beta[\gamma]}$.

Corollary \ref{corollary:tight.filters.of.infinite.type} says there are no tight filters other than the ultrafilters among the filters of infinite type. The next result classifies the tight filters of finite type.

\begin{proposition}
	\label{prop:tight.filters.of.finite.type}
	Suppose the accommodating family $\acf$ to be closed under relative complements and let $\xia$ be a filter of finite type. Then $\xia$ is a tight filter if and only if $\xi_{|\alpha|}$ is an ultrafilter and at least one of the following conditions hold:
	\begin{enumerate}[(a)]
		\item There is a net $\{\ftg{F}_{\lambda}\}_{\lambda\in\Lambda}\subseteq X_{\alpha}^{sink}$ converging to $\xi_{|\alpha|}$.
		\item There is a net $\{(t_{\lambda},\ftg{F}_{\lambda})\}_{\lambda\in\Lambda}$, where $t_{\lambda}$ is a letter in $\alf$ and $\ftg{F}_{\lambda}\in X_{\alpha t_{\lambda}}$ for each $\lambda\in\Lambda$, such that $\{f_{\alpha[t_{\lambda}]}(\ftg{F}_{\lambda})\}_{\lambda\in\Lambda}$ converges to $\xi_{|\alpha|}$ and for every $b\in\alf$ there is $\lambda_b\in\Lambda$ such that $t_{\lambda}\neq b$ for all $\lambda\geq\lambda_b$.
	\end{enumerate}
\end{proposition}
\begin{proof}
	We show the ``if'' part first. Suppose that $\xi_{|\alpha|}$ is an ultrafilter and that (a) holds. For each $\lambda\in\Lambda$, consider the filter $\eta^{\lambda}$ associated with the pair $(\alpha,\ftg{F}_{\lambda})$ (see Proposition \ref{prop:filter.from.finite.word.and.filter}). By Theorem \ref{thm:ultrafilters}, $\{\eta^{\lambda}\}_{\lambda\in\Lambda}$ is a net of ultrafilters in $\filt$. Furthermore, since $\{\eta^{\lambda}_{|\alpha|}\}_{\lambda\in\Lambda}$ converges to $\xi_{|\alpha|}$ in $X_\alpha$ and the functions $f_{\alpha_{1,m}[\alpha_{m+1,|\alpha|}]}$ are continuous, then $\{\eta^{\lambda}_{m}\}_{\lambda\in\Lambda}$ converges to $\xi_{m}$ for every $0\leq m\leq|\alpha|$. From this, it is easy to see that $\{\eta^{\lambda}\}_{\lambda\in\Lambda}$ converges to $\xia$, whence $\xia$ is tight, by Theorem \ref{thm:closure.of.ultrafilters}.
	
	Now, suppose that $\xi_{|\alpha|}$ is an ultrafilter and that (b) holds. For each $\lambda\in\Lambda$, consider an ultrafilter $\eta^{\lambda}$ containing the filter associated with the pair $(\alpha t_\lambda,\ftg{F}_{\lambda})$ and observe that $\eta^{\lambda}_{|\alpha|+1}=\ftg{F}_{\lambda}$. We claim that $\{\eta^{\lambda}\}_{\lambda\in\Lambda}$ converges to $\xia$. To see this, take $(\beta,B,\beta)\in E(S)$. We divide the proof in three cases. \\
	{\it Case 1:} Suppose $\beta$ is a beginning of (or equal to) $\alpha$. In this case, $\beta=\alpha_{1,m}$ for some $m\leq|\alpha|$. Since $\{\eta^{\lambda}_{|\alpha|}\}_{\lambda\in\Lambda}=\{f_{\alpha[t_{\lambda}]}(\eta^{\lambda}_{|\alpha|+1})\}_{\lambda\in\Lambda}$ converges to $\xi_{|\alpha|}$ and $f_{\alpha_{1,m}[\alpha_{m+1,|\alpha|}]}$ is continuous, then $\{\eta^{\lambda}_{m}\}_{\lambda\in\Lambda}$ converges to $\xi_m$. From this, it is easy to see there is $\lambda_0\in\Lambda$ such that for every $\lambda\geq\lambda_0$, $(\beta, B,\beta)\in\eta^{\lambda}$ if and only if $(\beta, B, \beta)\in\xia$. \\
	{\it Case 2:} Suppose $\beta$ and $\alpha$ are not comparable. In this case, $(\beta, B,\beta)\notin\xia$ and $(\beta, B,\beta)\notin\eta^{\lambda}$ for all $\lambda\in\Lambda$. \\
	{\it Case 3:} Suppose $\alpha$ is a beginning of (and not equal to) $\beta$. Write $\beta=\alpha\gamma$ and denote by $b$ the first letter of $\gamma$. By (b), we can choose $\lambda_b\in\Lambda$ such that $t_\lambda\neq b$ for all $\lambda\geq\lambda_b$. This says that $(\beta, B, \beta)\notin\eta^{\lambda}$ if $\lambda\geq\lambda_b$. Since $(\beta, B, \beta)\notin\xia$, we are done.
	
	Now, for ``only if'' side: suppose $\xia$ is a tight filter. By Corollary \ref{corollary:nontight.filters.of.finite.type}, $\xi_{|\alpha|}$ is an ultrafilter in $\acfra$. By Theorem \ref{thm:closure.of.ultrafilters}, we can choose a net $\{\eta^{\lambda}\}_{\lambda\in\Lambda}$ of ultrafilters in $E(S)$ which converges to $\xia$. Consider the following disjoint subsets of $\Lambda$:
	$$\begin{array}{l}
	\Lambda_1 = \{\lambda\in\Lambda \ | \ \mbox{the word of} \ \eta^{\lambda} \ \mbox{is a beginning of (and not equal to)} \ \alpha\}, \\
	\Lambda_2 = \{\lambda\in\Lambda \ | \ \mbox{the word of} \ \eta^{\lambda} \ \mbox{is equal to} \ \alpha\}, \\
	\Lambda_3 = \{\lambda\in\Lambda \ | \ \alpha \ \mbox{is a beginning of (and not equal to) the word of} \ \eta^{\lambda}\}, \\	
	\Lambda_4 = \{\lambda\in\Lambda \ | \ \mbox{the word of} \ \eta^{\lambda} \ \mbox{is not comparable with} \ \alpha\}. \end{array}$$
	Because $\{\eta^{\lambda}\}_{\lambda\in\Lambda}$ converges to $\xia$, $\Lambda_1$ and $\Lambda_4$ cannot be cofinal subsets of $\Lambda$. Since $\Lambda=\Lambda_1\cup\Lambda_2\cup\Lambda_3\cup\Lambda_4$, then $\Lambda_2$ or $\Lambda_3$ are cofinal. If $\Lambda_2$ is cofinal then, by Theorem \ref{thm:ultrafilters},  $\{\eta^{\lambda}_{|\alpha|}\}_{\lambda\in\Lambda_2}$ is a net in $X_{\alpha}^{sink}$. Furthermore, it is clear that $\{\eta^{\lambda}_{|\alpha|}\}_{\lambda\in\Lambda_2}$ converges to $\xi_{|\alpha|}$, establishing (a). If $\Lambda_3$ is cofinal, denote by $t_\lambda$ the first letter after $\alpha$ in the word of $\eta^{\lambda}$, for each $\lambda\in\Lambda_3$. Thus, $\{(t_{\lambda},\eta^{\lambda}_{|\alpha|+1})\}_{\lambda\in\Lambda_3}$ is a net such that $\eta^{\lambda}_{|\alpha|+1}\in X_{\alpha t_{\lambda}}$. Since $\{\eta^{\lambda}\}_{\lambda\in\Lambda_3}$ converges to $\xia$, then $\{f_{\alpha[t_{\lambda}]}(\eta^{\lambda}_{|\alpha|+1})\}_{\lambda\in\Lambda_3}=\{\eta^{\lambda}_{|\alpha|}\}_{\lambda\in\Lambda_3}$ converges to $\xi_{|\alpha|}$. This establishes the first part of (b).
	
	Finally, let $b\in\alf$ be given. If $\alpha b$ is not a labelled path, then $t_{\lambda}\neq b$ for every $\lambda\in\Lambda_3$. If $\alpha b$ is a labelled path, then $(\alpha b, r(\alpha b), \alpha b)\in E(S)$ is such that $(\alpha b, r(\alpha b), \alpha b)\notin \xia$. By the definition of convergence, there exists $\lambda_b\in\Lambda_3$ such that $(\alpha b, r(\alpha b), \alpha b)\notin \eta^{\lambda}$ for all $\lambda\geq\lambda_b$, which is equivalent to $t_{\lambda}\neq b$ for all $\lambda\geq\lambda_b$. This completes (b).	
\end{proof}

\begin{remark}
	\label{rmk:tight.filters.of.finite.type}
	The following conditions are easily seen to be equivalent (and therefore condition (b) in the above result can be replaced with its equivalent formulation below):
	\begin{enumerate}[(i)]
		\item There is a net $\{(t_{\lambda},\ftg{F}_{\lambda})\}_{\lambda\in\Lambda}$, where $t_{\lambda}$ is a letter in $\alf$ and $\ftg{F}_{\lambda}\in X_{\alpha t_{\lambda}}$ for each $\lambda\in\Lambda$, such that $\{f_{\alpha[t_{\lambda}]}(\ftg{F}_{\lambda})\}_{\lambda\in\Lambda}$ converges to $\xi_{|\alpha|}$ and for every $b\in\alf$ there is $\lambda_b\in\Lambda$ such that $t_{\lambda}\neq b$ for all $\lambda\geq\lambda_b$.
		\item There is a net $\{(\beta_{\lambda},\ftg{F}_{\lambda})\}_{\lambda\in\Lambda}$, where $\beta_{\lambda}$ is a word in $\awplus$ and $\ftg{F}_{\lambda}\in X_{\alpha \beta_{\lambda}}$ for each $\lambda\in\Lambda$, such that $\{f_{\alpha[\beta_{\lambda}]}(\ftg{F}_{\lambda})\}_{\lambda\in\Lambda}$ converges to $\xi_{|\alpha|}$ and for every $\gamma\in\awplus$ there is $\lambda_\gamma\in\Lambda$ such that $\beta_{\lambda}$ and $\gamma$ are not comparable for all $\lambda\geq\lambda_b$.
	\end{enumerate}
	Furthermore, if we give $\alf$ the discrete topology, the condition ``for every $b\in\alf$ there is $\lambda_b\in\Lambda$ such that $t_{\lambda}\neq b$ for all $\lambda\geq\lambda_b$'' is the same as saying the net $\{t_{\lambda}\}_{\lambda\in\Lambda}\subseteq\alf$ converges to infinity.
\end{remark}

The next result gives a more algebraic description for the tight filters of finite type in $E(S)$.

\begin{proposition}
	\label{prop:tight.filters.of.finite.type.algebraic}
	Suppose the accommodating family $\acf$ is closed under relative complements and let $\xia$ be a filter of finite type. Then $\xia$ is a tight filter if and only if $\xi_{|\alpha|}$ is an ultrafilter and for each  $A\in\xi_{|\alpha|}$ at least one of the following conditions hold:
	\begin{enumerate}[(a)]
		\item $\lbf(A\dgraph^1)$ given by (\ref{eqn:lae1}) is infinite.
		\item There exists $B\in\acfra$ such that $\emptyset\neq B\subseteq A\cap \dgraph^0_{sink}$.
	\end{enumerate}
\end{proposition}

\begin{proof}
	Let $\xia$ be a tight filter as in part (a) of Proposition \ref{prop:tight.filters.of.finite.type} and $A\in \xi_{|\alpha|}$ be given. If $\lbf(A\dgraph^1)$  is infinite there is nothing to prove. Suppose then that $\lbf(A\dgraph^1)$ is finite. Since $A\in\xi_{|\alpha|}$ there exists $\lambda$ such that $A\in\ft_{\lambda}\in X_{\alpha}^{sink}$. By the definition of $X_{\alpha}^{sink}$, for each $b\in \lbf(A\dgraph^1)$ there exists $C_b\in\ft_{\lambda}$ such that $r(C_b,b)=\emptyset$. Since $\lbf(A\dgraph^1)$ is finite, if $C$ is the intersection of all such $C_b$ then $C\in\ft_{\lambda}$. It follows that $r(A\cap C,b)=\emptyset$ for all $b\in\alf$ and hence $A\cap C\subseteq \dgraph^0_{sink}$. Defining $B=A\cap C$, we have that $B\in\ft_{\lambda}$ so that $B\in\acfra$ and clearly $\emptyset\neq B\subseteq A\cap \dgraph^0_{sink}$.
	
	Now let $\xia$ be a tight filter as in part (b) of Proposition \ref{prop:tight.filters.of.finite.type} and $A\in \xi_{|\alpha|}$ be given. There exists $\lambda_0$ such that $A\in f_{\alpha[t_{\lambda}]}(\ftg{F}_{\lambda})$ for all $\lambda\geq\lambda_0$. That means that $r(A,t_{\lambda})\neq\emptyset$ for all $\lambda\geq\lambda_0$. Observe that there are infinitely many different letters in the set of all $t_{\lambda}$ for $\lambda\geq\lambda_0$. It follows that $\lbf(A\dgraph^1)$ is infinite.
	
	For the converse, we use Proposition \ref{prop:tight.filter.characterization}, and observe it is enough to consider $x\in\xia$ of the form $x=(\alpha,A,\alpha)$: indeed, if $x=(\alpha',A,\alpha')\in\xia$ and $Z$ is a finite cover for $x$ observe that $\alpha=\alpha'\alpha''$; if we define $x'=(\alpha,r(A,\alpha''),\alpha)$ and $Z'=\{zx'\ |\ z\in Z\}\setminus\{0\}$ (which is a finite cover for $x'$), it is easy to see that if $Z'\cap \xia\neq\emptyset$ then $Z\cap\xia\neq\emptyset$.
	
	Consider therefore $x=(\alpha,A,\alpha)\in\xia$ and le $Z$ be a finite cover for $x$. We claim that there exists $\emptyset\neq D\in \acfra$ such that $(\alpha,D,\alpha)\in Z$. Indeed, if $\lbf(A\dgraph^1)$ is infinite then there are infinitely many elements of the form $y_i=(\alpha b_i,r(A,b_i),\alpha b_i)$ with distinct $b_i$'s in $\lset{x}$. For each $y_i$ there exists $z=(\alpha\beta,C,\alpha\beta)\in Z$ such that $y_iz\neq 0$. Since $Z$ is finite and $\alpha\beta$ and $\alpha b_i$ are comparable, there must exist $(\alpha,D,\alpha)\in Z$ for some $\emptyset\neq D\in \acfra$. If there exists $B\in\acfra$ such that $\emptyset\neq B\subseteq A\cap \dgraph^0_{sink}$ then $y=(\alpha,B,\alpha)\in\lset{x}$. Since $B\subseteq \dgraph^0_{sink}$, for an element of $Z$ to intersect $y$ it has to be of the form $(\alpha,D,\alpha)\in Z$ for some $\emptyset\neq D\in \acfra$.
	
	Now let $D_1,\ldots,D_n$ be all the sets in $\acfra$ such that $(\alpha,D_i,\alpha)\in Z$. As we have seen above, $n\geq 1$. Observe that $D_1\cup\cdots\cup D_n\in\xi_{|\alpha|}$. If not, since $\xi_{|\alpha|}$ is an ultrafilter there would exist $C\in\xi_{|\alpha|}$ such that $C\cap (D_1\cup\cdots\cup D_n\in\xi_{|\alpha|})=\emptyset$. Now $A\cap C\in\xi_{|\alpha|}$ so that it satisfies (a) or (b). Arguing as above it can be shown that there exists $D$ such that $(\alpha,D,\alpha)$ is in $Z$ and intersects $(\alpha,A\cap C,\alpha)$ which is a contradiction. Since $\xi_{|\alpha|}$ is an ultrafilter, it is a prime filter so that $D_i\in\xi_{|\alpha|}$ for some $i=1,\ldots,n$. It follows that $\xia\cap Z\neq\emptyset$.
\end{proof}

Note that, in using the proposition above, it is enough to verify condition (b) only in the case $\lbf(A\dgraph^1)$ is finite.

Joining Corollary \ref{corollary:tight.filters.of.infinite.type} and the previous proposition, we get the main result of this section.

\begin{theorem}
	\label{thm:tight.filters.in.es}
	Let $\lspace$ be a labelled space which is weakly left resolving and whose accommodating family $\acf$ is closed under relative complements, and let $S$ be its associated inverse semigroup. Then the tight filters in $E(S)$ are:
	\begin{enumerate}[(i)]
		\item The ultrafilters of infinite type.
		\item The filters of finite type $\xia$ such that $\xi_{|\alpha|}$ is an ultrafilter in $\acfra$ and for each  $A\in\xi_{|\alpha|}$ at least one of the following conditions hold:
		\begin{enumerate}[(a)]
			\item $\lbf(A\dgraph^1)$ is infinite.
			\item There exists $B\in\acfra$ such that $\emptyset\neq B\subseteq A\cap \dgraph^0_{sink}$.
		\end{enumerate}
	\end{enumerate}
\end{theorem}

\begin{example}\label{example:tight.spectrum.dgraph}
	Let $\dgraphuple$ be any graph and define a labelled graph by setting $\alf=\dgraph^1$ and $\lbf:\alf\to\dgraph^1$ the identity map. Let $\acf$ be the sets of all finite subsets of $\dgraph^0$, then $\acf$ is an accommodating family for $\lgraph$ that is closed under relative complements. We prove that for the labelled space $\lspace$, $\ftight$ is homeomorphic to the boundary path space of the graph $\partial \dgraph$ as in \cite{MR3119197}.
	
	Recall that $\partial\dgraph = \dgraph^{\infty}\cup\{\lambda\in\dgraph^*\,|\,r(\lambda)\text{ is singular}\}$. On $\dgraph^* \cup\dgraph^{\infty}$ we give the initial topology induced by the map $\theta:\dgraph^* \cup\dgraph^{\infty}\to \{0,1\}^{\dgraph^*}$ defined by $\theta(\lambda)(\mu)=1$ if $\mu$ is a beginning of $\lambda$ and $0$ otherwise. We then give the subspace topology on $\partial\dgraph$.
	
	Now, observe that for $\alpha\in\awplus$, $\acfra=\{\emptyset,r(\alpha)\}$ so that it admits only one ultrafilter, which is $\{r(\alpha)\}$. And in $\acf_{\eword}=\acf$ the ultrafilters are exactly the principal filters $\usetr{\{v\}}{\acf}$ where $v\in\dgraph^0$. It follows that for each $\alpha\in\awplus\cup\awinf$ there is only one filter $\xi$ in $E$ with word $\alpha$. For the empty word the filters $\xi^{\eword}$ such that $\xi_0$ is an ultrafilter are in a bijective correspondence with vertices $v\in\dgraph^0$.
	
	By Theorem \ref{thm:tight.filters.in.es} we see that the tight filters are those associated with an infinite word $\alpha\in\awinf$, those associated with $\alpha\in\awplus$ such that $r(\alpha)$ is singular or $\xi^{\eword}$ where $\xi_0$ is the principal filter generated by a singular vertex. It follows that there is a bijection between $\ftight$ and $\partial\dgraph$ which is a homeomorphism since both topologies are essentially given by pointwise convergence.

\end{example}

\begin{proposition}\label{prop:tight.spectrum.is.boundary.path.space}
	Let $\lgraph$ be a left resolving labelled graph such that $\dgraph^0$ is a finite set and let  $\acf=\powerset{\dgraph^0}$. Then the tight spectrum $\ftight$ of the labelled space $\lspace$ is homeomorphic to the boundary path space $\partial\dgraph$ of the graph $\dgraph$.
\end{proposition}

\begin{proof}
	Let $\lambda=\lambda_1\lambda_2\cdots\lambda_m$ be a finite path in $\partial\dgraph$, set $\alpha=\lbf(\lambda)$ and, for each $0\leq n\leq m$, let $\xi^{\alpha}_n=\usetr{\{r(\lambda_n)\}}{\acfrg{\alpha_{1,n}}}$ (which for the given accomodating family means simply the set of all subsets of $r(\alpha_{1,n})$ that contain the element $r(\lambda_n)$). When $n=0$, we set $\lambda_0=s(\lambda)$. Similarly, if $\lambda=\lambda_1\lambda_2\cdots$ is an infinite path in $\partial\dgraph$, set $\alpha=\lbf(\lambda)$ and let $\xi^{\alpha}_n=\usetr{\{r(\lambda_n)\}}{\acfrg{\alpha_{1,n}}}$ for $n\geq0$. In both cases, we can consider the pair $(\alpha,\{\xi^{\alpha}_n\}_{n\geq0})$ associated with $\lambda$.
	
	We claim that $\{\xi^{\alpha}_n\}_{n\geq0}$ is a complete family for $\alpha$. Indeed, fix $n\geq0$ (with $n<|\lambda|$ if $\lambda$ is finite) and let $A\in\xi^{\alpha}_n$. Since $r(\lambda_n)\in A$, then $r(\lambda_{n+1})\in r(A,\alpha_{n+1})$, hence
	\[\xi^{\alpha}_n\subseteq\{A\in\acfrg{\alpha_{1,n}} \ | \ r(A,\alpha_{n+1})\in\xi^{\alpha}_{n+1}\}.\]
	
	On the other hand, let $A\in\acfrg{\alpha_{1,n}}$ be such that $r(A,\alpha_{n+1})\in\xi^{\alpha}_{n+1}$. Therefore there exists $v\in A$ and $e\in\dgraph^1$ such that $s(e)=v$, $r(v)=r(\lambda_{n+1})$ and $\lbf(e)=\alpha_{n+1}$. The labelled graph is left resolving, so we must have $e=\lambda_{n+1}$. This says $r(\lambda_n)=s(\lambda_{n+1})=v\in A$, showing that
	\[\{A\in\acfrg{\alpha_{1,n}} \ | \ r(A,\alpha_{n+1})\in\xi^{\alpha}_{n+1}\}\subseteq\xi^{\alpha}_n\] and that the family $\{\xi^{\alpha}_n\}_{n\geq0}$ is complete for $\alpha$, as claimed. Also, observe that $\xi^{\alpha}_n$ is an ultrafilter in $\acfrg{\alpha_{1,n}}$ for all $n$.
	
	Let $\xia$ be the filter in $E(S)$ associated with the pair $(\alpha,\{\xi^{\alpha}_n\}_{n\geq0})$. We show next that $\xia$ is tight. There are three cases to consider. \\
	{\it Case 1:} $\lambda$ is an infinite path. In this case, $\xia$ is a filter of infinite type and it is clear that it is an ultrafilter. It follows from Theorem \ref{thm:tight.filters.in.es} that $\xia$ is tight. \\
	{\it Case 2:} $\lambda$ is a finite path and $s^{-1}(r(\lambda))$ is an infinite set. In this case, $\xia$ is a filter of finite type. Let $A\in\xi^{\alpha}_{|\alpha|}$. Using that $r(\lambda)\in A$, that the labelled graph is left resolving and that $\dgraph^0$ is finite, we see that $\lbf(A\dgraph^1)=\infty$. It follows that $\xia$ is tight, by Theorem \ref{thm:tight.filters.in.es}. \\
	{\it Case 3:} $\lambda$ is a finite path and $r(\lambda)\in\dgraph^0_{sink}$. In this case, $\xia$ is a filter of finite type. For every $A\in\xi^{\alpha}_{|\alpha|}$, letting $B=\{r(\lambda)\}$ we have that  $B\in\acf$ and $\emptyset\neq B\subseteq A\cap\dgraph^0_{sink}$. Again Theorem \ref{thm:tight.filters.in.es} ensures $\xia$ is tight.
	
	We can thus define \[\newfd{\Phi}{\partial\dgraph}{\ftight}{\lambda}{\xia}.\]	
	By construction, $\Phi$ is clearly injective. Let us show surjectivity: let $\xia\in\ftight$. For each $n\geq0$ (with $n\leq|\alpha|$ if $\xia$ is of finite type) the set $r(\alpha_{1,n})$ is finite, and thus all ultrafilters in $\acfrg{\alpha_{1,n}}=\powerset{r(\alpha_{1,n})}$ are principal; in particular, there exists $v_n\in r(\alpha_{1,n})$ such that $\xi^{\alpha}_n=\usetr{\{v_n\}}{\acfrg{\alpha_{1,n}}}$.
	
	Observe that for each $n\geq1$, there exists a unique $\lambda_n\in\dgraph^1$ such that $s(\lambda)=v_{n-1}$, $r(\lambda_n)=v_n$ and $\lbf(\lambda)=\alpha_n$. Indeed, using that $\{\xi^{\alpha}_n\}_{n\geq0}$ is a complete family, we must have $r(\{v_{n-1}\},\alpha_n)\in\xi^{\alpha}_n$. Hence, there is $\lambda_n\in\dgraph^1$ such that $s(\lambda)=v_{n-1}$, $r(\lambda_n)=v_n$ and $\lbf(\lambda)=\alpha_n$. The uniqueness of $\lambda_n$ follows from the fact that the labelled graph is left resolving.
	
	Thus, we have constructed a path $\lambda=\lambda_1\lambda_2\cdots$ such that $\lbf(\lambda)=\alpha$. Let us see that $\lambda\in\partial\dgraph$: if $|\alpha|=\infty$, there is nothing to do. Suppose $|\alpha|<\infty$ and consider $A=\{r(\lambda)\}=\{v_{|\alpha|}\}$. If $\lbf(A\dgraph^1)=\infty$, then $s^{-1}(r(\lambda))$ is clearly infinite, whence $\lambda\in\partial\dgraph$. If $\lbf(A\dgraph^1)<\infty$, then there is $B\in\acf$ such that $\emptyset\neq B\subseteq A\cap\dgraph^0_{sink}$. Since $A$ is unitary, that means $B=A=\{r(\lambda)\}$ and $r(\lambda)\in\dgraph^0_{sink}$, which shows that $\lambda\in\partial\dgraph$. By construction, we have $\Phi(\lambda)=\xia$, establishing that $\Phi$ is a bijection.
	
	To see that $\Phi$ is a homeomorphism, observe that under the conditions stated above, a net $\{\xi_{(j)}\}_{j\in J}$ converges to $\xi$ in $\ftight$ if and only if for all $(\beta,\{v\},\beta)\in E(S)$, there exists $j_0\in J$ such that for all $j\geq j_0$,
	\[(\beta,\{v\},\beta)\in\xi_{(j)} \ \ \Longleftrightarrow \ \  (\beta,\{v\},\beta)\in\xi.\]
		
	Also observe that a net $\{\lambda_{(j)}\}_{j\in J}$ in $\partial\dgraph$ converges to $\lambda\in\partial\dgraph$ if and only if for all $\mu\in\dgraph^*$ there exists $j_0\in J$ such that for all $j\geq j_0$, $\lambda_i$ begins with $\mu$ if and only if $\lambda$ begins with $\mu$ (see example \ref{example:tight.spectrum.dgraph} above).
	
	In the same way as was done above, the map between $\dgraph^*$ and the set $\{(\beta,A,\beta)\in E(S) \ | \ A=\{v\} \ \mbox{is singleton}\}$ given by $\lambda\mapsto (\lbf(\lambda),\{r(\lambda)\},\lbf(\lambda))$ is a bijection. From these observations, it is not hard to see that $\Phi$ is a homeomorphism.
\end{proof}

\begin{example} This example shows that, under the hypotheses of the previous result, if we consider an accommodating family other than $\powerset{\dgraph^0}$ then the boundary path space of the graph may no longer be identified with the tight spectrum as was done in the proof above. For instance, consider the labelled graph below:
	\begin{center}
		\resizebox{5cm}{!}
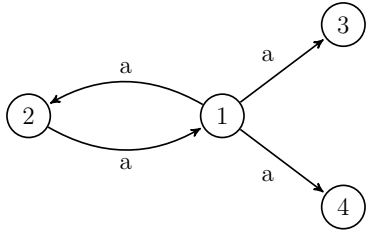
{
			\begin{tikzpicture}[->,>=stealth',shorten >=1pt,auto,thick,main node/.style={circle,draw,thick,font=\bfseries}]
			\node[main node] (1) {$1$};
			\node[main node] (2) [left=2.5cm of 1] {$2$};
			\node[main node] (3) [above right=1cm and 1.5cm of 1] {$3$};
			\node[main node] (4) [below right=1cm and 1.5cm of 1] {$4$};
			
			\draw[->] (2) to [bend right] node [below] {a} (1);
			\draw[->] (1) to [bend right] node [above] {a} (2);
			\draw[->] (1) to node [above left] {a} (3);
			\draw[->] (1) to node [below left] {a} (4);
			\end{tikzpicture}
		}
	\end{center}
	
	It is readily verified that \[\acf=\{\emptyset,\{1\},\{3\},\{1,3\},\{2,4\},\{1,2,4\},\{2,3,4\},\{1,2,3,4\}\}\] is an accommodating family for the given labelled graph, that is closed under relative complements, and that the resulting labelled space is weakly left resolving (indeed, the labelled graph is left resolving, even).
	
	We can see that for each $n\geq0$, the boundary path space of the underlying graph has two elements of length $n$. Let us show that, for the only labelled path of length $n$, there is a single tight filter associated with it.
	
	Note that the only ultrafilters in $\acf$ are $\ftg{G}_1=\usetr{\{1\}}{\acf}$, $\ftg{G}_2=\usetr{\{2,4\}}{\acf}$ and $\ftg{G}_3=\usetr{\{3\}}{\acf}$. Also, $r(\alpha)=\dgraph^0$ and therefore $\acfra=\acf$ for all $\alpha\in\awstar$. Finally, for $n\geq 0$ observe that \[f_{a^n[a]}(\ftg{G}_1)=\ftg{G}_2,\ f_{a^n[a]}(\ftg{G}_2)=\ftg{G}_1, f_{a^n[a]}(\ftg{G}_3)=\ftg{G}_1.\]
	
	From this and using Theorem \ref{thm:tight.filters.in.es}, we see that the tight filters for this labelled space are classified as follows:
	
	There are two ultrafilters of infinite type. The first is associated with the complete family $\{\ft_n\}_n$ where $\ft_n$ equals $\ftg{G}_1$ if $n$ is odd and $\ftg{G}_2$ if n is even. The second ultrafilter is the opposite: $\ft_n$ equals $\ftg{G}_1$ if $n$ is even and $\ftg{G}_2$ if n is odd.
	
	There is a single tight filter for each $\alpha\in\awstar$. This happens for two reasons: first, condition $(ii)(a)$ of Theorem \ref{thm:tight.filters.in.es} does not occur here since the alphabet $\alf$ consists of a single letter; second, $\ftg{G}_3$ is the only ultrafilter in $\acf$ that satisfies condition $(ii)(b)$. 
	
	Thus, a filter $\xia$ in $E(S)$ is tight if and only if $\xi_{|\alpha|}=\ftg{G}_3$. Additionally, if $|\alpha|=n+1$ (that is, $\alpha=a^{n+1}$) with $n\geq 0$ then for each $k\leq n$ we have that
	\[\xi_k=\left\{\begin{array}{ll}
	\ftg{G}_1, & \text{if}\ n\ \text{and}\ k\ \text{have different parities},\\
	\ftg{G}_2, & \text{if}\ n\ \text{and}\ k\ \text{have the same parity}.
	\end{array}\right. \]
	
	The forced relations between the ultrafilters that produce the complete families can be seen with the following directed graph, (where the arrows indicate maps of type $f_{a^r[a]}$, $r\geq 0$):
	\begin{center}
		\resizebox{6cm}{!}
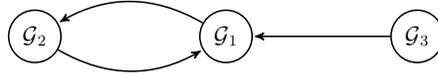
{
			\begin{tikzpicture}[->,>=stealth',shorten >=1pt,auto,thick,main node/.style={circle,draw,thick,font=\bfseries}]
			\node[main node] (1) {$\ftg{G}_1$};
			\node[main node] (2) [left=2.5cm of 1] {$\ftg{G}_2$};
			\node[main node] (3) [right=2.5cm of 1] {$\ftg{G}_3$};
			
			\draw[->] (2) to [bend right] node [below] {} (1);
			\draw[->] (1) to [bend right] node [above] {} (2);
			\draw[->] (3) to node [above] {} (1);
			\end{tikzpicture}
		}
	\end{center}
	
\end{example}

\begin{example}
	In this example, we show that the tight spectrum of a labelled space does not depend only on its labelled paths. Consider the two labelled graphs $(\dgraphg{E}_1,\lbfg{L}_1)$ and $(\dgraphg{E}_2,\lbfg{L}_2)$ given respectively by
	
	\resizebox{12cm}{!}
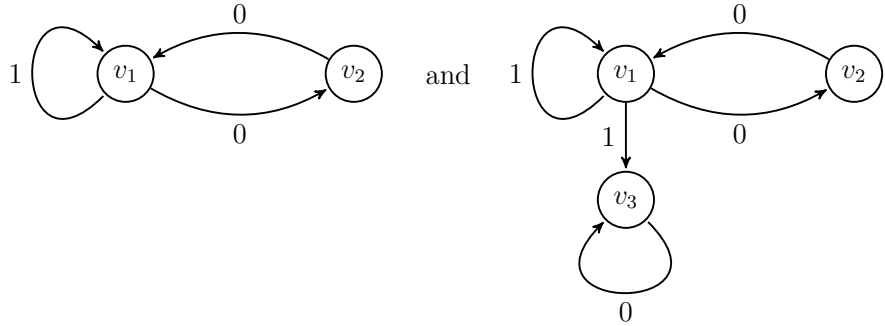
{
		\begin{tikzpicture}[->,>=stealth',shorten >=1pt,auto,thick,main node/.style={circle,draw,thick,font=\bfseries}]
		\node[main node] (1) {$v_1$};
		\node[main node] (2) [right=2.5cm of 1] {$v_2$};
		\node[main node] (3) [below=1cm of 1] {$v_3$};
		
		\node (6) [left = 1.7cm of 1] {\text{and}};
		\node[main node] (4) [left=6.5cm of 1] {$v_1$};
		\node[main node] (5) [right=2.5cm of 4] {$v_2$}; 
		
		\Loop[dist=2cm,dir=WE,label=$1$,labelstyle=left](1)  
		\Loop[dist=2cm,dir=SO,label=$0$,labelstyle=below](3)
		
		\draw[->] (1) to node [left] {1} (3);
		\draw[->] (2) to [bend right] node [above] {0} (1);
		\draw[->] (1) to [bend right] node [below] {0} (2);
		
		\Loop[dist=2cm,dir=WE,label=$1$,labelstyle=left](4) 
		
		\draw[->] (5) to [bend right] node [above] {0} (4);
		\draw[->] (4) to [bend right] node [below] {0} (5);		
		\end{tikzpicture}
	}
	
	It is easy to see that both labelled graphs produce the same labelled paths. As accommodating families for each, choose $\acf_1=\powerset{\dgraph_1^0}$ and $\acf_2=\powerset{\dgraph_2^0}$, respectively. Proposition \ref{prop:tight.spectrum.is.boundary.path.space} says the tight spectra of $(\dgraphg{E}_1,\lbfg{L}_1,\acf_1)$ and $(\dgraphg{E}_2,\lbfg{L}_2,\acf_2)$ are homeomorphic to $\partial\dgraph_1$ and $\partial\dgraph_2$, respectively.
	
	However, $\partial\dgraph_1$ has no isolated points, as can be readily verified, whereas the infinite path determined by the loop at vertex $v_3$ is an isolated point for $\partial\dgraph_2$; therefore, they are not homeomorphic.	

\end{example}

\bibliographystyle{abbrv}
\bibliography{labelledgraphs_ref}

\begin{thebibliography}{10}

\bibitem{2012arXiv1203.3072B}
T.~{Bates}, T.~{Meier Carlsen}, and D.~{Pask}.
\newblock {{$C\sp *$}-algebras of labelled graphs III - K-theory computations}.
\newblock {\em ArXiv e-prints}, Mar. 2012.

\bibitem{MR2304922}
T.~Bates and D.~Pask.
\newblock {$C\sp *$}-algebras of labelled graphs.
\newblock {\em J. Operator Theory}, 57(1):207--226, 2007.

\bibitem{MR2542653}
T.~Bates and D.~Pask.
\newblock {$C\sp *$}-algebras of labelled graphs. {II}. {S}implicity results.
\newblock {\em Math. Scand.}, 104(2):249--274, 2009.

\bibitem{2011arXiv1106.1484B}
T.~{Bates}, D.~{Pask}, and P.~{Willis}.
\newblock {Group actions on labeled graphs and their C*-algebras}.
\newblock {\em ArXiv e-prints}, June 2011.

\bibitem{MR648287}
S.~Burris and H.~P. Sankappanavar.
\newblock {\em A course in universal algebra}, volume~78 of {\em Graduate Texts
  in Mathematics}.
\newblock Springer-Verlag, New York-Berlin, 1981.

\bibitem{MR0161813}
H.~P. Doctor.
\newblock The categories of {B}oolean lattices, {B}oolean rings and {B}oolean
  spaces.
\newblock {\em Canad. Math. Bull.}, 7:245--252, 1964.

\bibitem{MR2419901}
R.~Exel.
\newblock Inverse semigroups and combinatorial {$C\sp \ast$}-algebras.
\newblock {\em Bull. Braz. Math. Soc. (N.S.)}, 39(2):191--313, 2008.

\bibitem{MR2184052}
C.~Farthing, P.~S. Muhly, and T.~Yeend.
\newblock Higher-rank graph {$C^*$}-algebras: an inverse semigroup and groupoid
  approach.
\newblock {\em Semigroup Forum}, 71(2):159--187, 2005.

\bibitem{MR2457327}
A.~E. Marrero and P.~S. Muhly.
\newblock Groupoid and inverse semigroup presentations of ultragraph
  {$C^*$}-algebras.
\newblock {\em Semigroup Forum}, 77(3):399--422, 2008.

\bibitem{MR1962477}
A.~L.~T. Paterson.
\newblock Graph inverse semigroups, groupoids and their {$C^\ast$}-algebras.
\newblock {\em J. Operator Theory}, 48(3, suppl.):645--662, 2002.

\bibitem{MR584266}
J.~Renault.
\newblock {\em A groupoid approach to {$C^{\ast} $}-algebras}, volume 793 of
  {\em Lecture Notes in Mathematics}.
\newblock Springer, Berlin, 1980.

\bibitem{MR1507106}
M.~H. Stone.
\newblock Postulates for {B}oolean {A}lgebras and {G}eneralized {B}oolean
  {A}lgebras.
\newblock {\em Amer. J. Math.}, 57(4):703--732, 1935.

\bibitem{MR3119197}
S.~B.~G. Webster.
\newblock The path space of a directed graph.
\newblock {\em Proc. Amer. Math. Soc.}, 142(1):213--225, 2014.

\end{thebibliography}

\end{document}